\documentclass[psamsfonts]{amsart}  
\usepackage{amssymb,tikz}
\usepackage[T1]{fontenc}
\usepackage[T1]{fontenc}
\usepackage[utf8]{inputenc}
\usepackage{bera}
\usepackage[portrait,margin=3cm]{geometry} 
\usepackage[mathscr]{euscript}
\usepackage[colorlinks=true,linkcolor=blue,citecolor=blue,urlcolor=blue]{hyperref} 
\usepackage{amsmath,amssymb,amsthm,amsfonts,amsbsy,latexsym,dsfont,color,graphicx,enumitem}
\usepackage[numeric,initials,nobysame]{amsrefs} 

\usepackage{caption}
\usepackage{subcaption}

\usepackage{amssymb,bbm,ulem}

\usepackage{enumitem}
\usepackage{color,comment}

\newcommand{\Z}{\mathbb{Z}}

\newcommand{\R}{\mathbb{R}}

\newcommand{\nH}{{H^{\alpha/2}(\R^d)}}

\newtheorem{thm}{Theorem}[section]
\newtheorem*{thm*}{Theorem}

\newtheorem{prop}[thm]{Proposition}

\newtheorem{rem}[thm]{Remark}

\theoremstyle{definition}
\newtheorem{df}[thm]{Definition}

\newtheorem{assumption}{Assumption}[section]
\newtheorem{example}[assumption]{Example}
\newtheorem{corollary}[assumption]{Corollary}
\newtheorem{lemma}[assumption]{Lemma}

\newtheorem*{claim*}{Claim}

\newtheorem*{acknowledgement}{Acknowledgement}

\numberwithin{equation}{section}

\DeclareMathOperator*{\esup}{ess\,sup}

\begin{document}

\title[]{
Survival probability for jump processes in unbounded domains on metric measure spaces
}

\author[Mariano]{Phanuel Mariano{$^\dag$}}
\thanks{\footnotemark {$\dag$} Research was supported in part by  NSF Grant DMS-2316968.}
	\address{Department of Mathematics\\
		Union College\\
		Schenectady, NY 12308,  U.S.A.}
	\email{marianop@union.edu}

\author[Wang]{Jing Wang{$^{\ddag}$}}
\thanks{\footnotemark {$\ddag$} Research was supported in part by NSF Grant DMS-2246817}
\address{ Department of Mathematics\\
Purdue University\\
West Lafayette, IN 47907,  U.S.A.}
\email{jingwang@purdue.edu}

\keywords{exit times, spectral asymptotics, heat kernel, symmetric stable process, jump process, metric measure space, horn-shaped domains, unbounded domains}
\subjclass{Primary 60J76, 35P15, 60J45; Secondary   58J65, 35J25, 35K08}


\begin{abstract}  

We study the large time behavior of the survival probability $\mathbb{P}_x\left(\tau_D>t\right)$ for symmetric jump processes in unbounded domains with a positive bottom of the spectrum. We prove asymptotic upper and lower bounds with explicit constants in terms of the bottom of the spectrum $\lambda(D)$. Our main result applies to symmetric jump processes in general metric measure spaces. For $\alpha$-stable processes in unbounded uniformly $C^{1,1}$ domains, our results provide a probabilistic interpretation and an equivalent geometric condition 
for $\lambda(D)>0$. In the case of increasing horn-shaped domains, the exponential rate of decay for the survival probability is sharp. We also present examples of unbounded domains where our results apply.
\end{abstract}

\maketitle


\section{Introduction and Main Results}\label{sec:1}

For $0<\alpha\leq 2$, we consider the $d-$dimensional \textbf{symmetric $\alpha-$stable
process}, which is the L\'evy process $\left\{ X_{t}\right\} $
whose characteristic function is given by 
\[
\mathbb{E}\left[e^{i\xi\cdot X_{1}}\right]=e^{-\left|\xi\right|^{\alpha}},\xi\in\mathbb{R}^{d}.
\]
When $\alpha=2$, then $X_{t}$ is a Brownian motion running twice the
usual speed, which we denote by $W_t$. We can write $X_{t}$ in distribution as 
\begin{equation}
X_{t}=W_{S_{t}},\label{eq:subbordinated-process}
\end{equation}
where $S_{t}$, $t\ge0$ is an $\frac{\alpha}{2}$-stable subordinator independent of $W_t$, $t\ge0$.  The infinitesimal generator of $X_t$ is given by the fractional Laplacian $-\left(-\Delta\right)^{\alpha/2}$  on $\mathbb{R}^d$, which is a 
non-local operator.

For any open set $D\subset \R^d$, define the \textbf{first exit time}
from $D$, to be 
\begin{align}\label{eq-tau-D}
    \tau_{D}:=\inf\left\{ t>0:X_{t}\notin D\right\} .
\end{align}
The distribution function $\mathbb{P}_{x}\left(\tau_{D}>t\right)$ of $\tau_D$ is also known as the survival probability of $X_t$. 
Let $\mathcal{L}_{D}^X:=\left(-\Delta\right)^{\alpha/2}\mid_D$  be the Dirichlet fractional Laplacian. It is defined as the fractional Laplacian on $D$ with zero exterior condition.
When $D$ is bounded, it is known that $\mathcal{L}_{D}^X$ has a discrete spectrum when acting on $L^{2}\left(D\right)$. The large time (asymptotic) behavior of the survival probability has been obtained in \cite{Chen-Song-1997,Kulczycki-1998} using the spectral expansion of the heat kernel corresponding to $\mathcal{L}_{D}^X$:
\[
\lim_{t\to\infty}\frac{1}{t}\log\mathbb{P}_{x}\left(\tau_{D}>t\right)=-\lambda_{1}\left(D\right),
\]
where $\lambda_{1}\left(D\right)$ is the first Dirichlet eigenvalue
corresponding to $\mathcal{L}_{D}^{X}$. However, such spectral method can not be applied to the case of unbounded domains, where  the generator $\mathcal{L}^X_{D}$ 
may not have a discrete spectrum. 

The purpose of this paper is to study the large time behavior of the survival probability $\mathbb{P}_{x}\left(\tau_{D}>t\right)$ of $\tau_D$ for symmetric jump-processes in an \textbf{unbounded} domain. 
For $0<\alpha<2$, let $\lambda(D)$ be the bottom of the spectrum which can be defined as 
\begin{equation}\label{eq-lamdD}
\lambda\left(D\right):=\inf\text{spec}\left(\mathcal{L}_{D}^{X}\right)=\inf_{u\in C_{0}^{\infty}\left(D\right)\backslash\left\{ 0\right\} }\frac{\left[u\right]_{W^{\frac{\alpha}{2},2}\left(\mathbb{R}^{d}\right)}^{2}}{\left\Vert u\right\Vert _{L^{2}\left(D\right)}^{2}},
\end{equation}
where $\left[u\right]_{W^{\frac{\alpha}{2},2}\left(\mathbb{R}^{d}\right)}^{2}$
is the Gagliardo--Slobodecki\u{\i} seminorm given by 
\[
\left[u\right]_{W^{\frac{\alpha}{2},2}\left(\mathbb{R}^{d}\right)}^{2}=\frac{c_{d,\alpha}}{2}\int\int_{\mathbb{R}^{d}\times\mathbb{R}^{d}}\frac{\left|u\left(x\right)-u\left(y\right)\right|^{2}}{\left|x-y\right|^{d+\alpha}}dxdy,
\]
where $c_{d,\alpha}=2^{\alpha}\pi^{-d/2}\Gamma\left(\frac{d+\alpha}{2}\right)\left|\Gamma\left(-\frac{\alpha}{2}\right)\right|^{-1}$. If $\lambda(D)=0$, then $\mathbb{P}_{x}\left(\tau_{D}>t\right)$ does not decay exponentially as $t\to\infty$. In such cases, several results in the literature establish the integrability of exit times and obtain asymptotic estimates of the survival probability of non-local processes in unbounded domains  (see \cite{Banuelos-Bogdan-2004,Banuelos-Bogdan-2005,Mendez-Hernandez-2002,Mendez-Hernandez-2007,DeBlassie-1990,Mariano-Panzo-2020}). For Brownian motion in particular, classical results (such as \cite{Banuelos-Carrol1994,Banuelos-DeBlassie-Smits-2001,Lifshits-Shi-2002,Li-2003,DeBlassie-Smits-2005,Berg-2003,Collet-etal-2000}) address these estimates in unbounded domains.

In the case of an unbounded domain $D$ with $\lambda(D)>0$, asymptotic estimates of the survival probability have been obtained for local processes. For instance, B. Simon  (\cite{Simon-1981,Simon-1981b}) showed  that for Brownian motion, or more generally elliptic diffusions,
\begin{equation}
\lim_{t\to\infty}\frac{\log\mathbb{P}_{x}\left(\tau_{D}>t\right)}{t}=-\lambda\left(D\right), \quad x\in D\label{eq:coj-main}
\end{equation}
for any unbounded domain $D\subset \mathbb{R}^d$.  Many related results for local processes have been established since, including in the works of  \cite{Vogt-2019a,Ouhabaz-2006,Sznitman-1998,Bakharev-Nazarov-JFA-2021,Lifshits-Nazarov-2019}. More recently, \cite{Mariano-Wang-2022} extended such estimates to general Hunt processes satisfying sub-Gaussian bounds in unbounded domains on metric measure spaces.

However, there remains a gap in the literature concerning the limiting behavior of the survival probability of a non-local process in unbounded domains with $\lambda(D)>0$. Although it is widely believed that \eqref{eq:coj-main} still holds in this setting,  to the author's knowledge, no proof has been established in the literature, not even for symmetric $\alpha-$stable processes on $\mathbb{R}^{d}$. Several works  have addressed related topics, such as large-time asymptotics, heat kernel bounds and explicit bounds for exit times of non-local processes (see the following works to name a few, \cite{Chen-Tokle-2011,Chen-Kim-Song-2010,Kim-Song-2014,Siudeja-2006, Chen-Kim-Wang-2022,Banuelos-Latala-Mendez-2001,Bogdan-Grzywny-2010,Barrier-Bogdan-Grzywny-Ryznar-2015,Kwasnicki-2009,Giorgi-Smits-2010,Panzo2021,Chen-Kim-Wang-2019}). However, most of these results rely on bounded domains or assume certain regularity conditions on the   boundary of the domain.

In this paper, we obtain both upper and lower bounds for the survival probability of a symmetric jump process in an unbounded domain of positive bottom of spectrum. Our result is derived in a general setting of metric measure space, see Theorem \ref{thm:MainTheorem1}. For clarity and readability we present the result here first in a simplified version for the setting of  $\R^d$.

\begin{thm}\label{Main:Thm:StableProc}
Let $X_{t}$ be an $\alpha-$stable process in a domain $D\subset\mathbb{R}^{d}$, $0<\alpha<2$. Let $\tau_{D}$ be its first exit time of $D$ as given in \eqref{eq-tau-D}. Suppose the bottom of the spectrum $\lambda\left(D\right)$
of $\mathcal{L}_{D}^X=\left(-\Delta\right)^{\alpha/2}\mid_{D}$ is positive. For every $\epsilon\in\left(0,1\right)$
 there exists a constant $C_{\epsilon}>0$
such 
\begin{equation}
\mathbb{P}_{x}\left(\tau_{D}>t\right)\leq C_{\epsilon}\exp\left(-\frac{1-\epsilon}{1+\frac{d}{4\alpha}}\lambda\left(D\right)t\right),\quad x\in D.\label{eq:Main-Simplified-Upper-1}
\end{equation}
Moreover, we have that for any $x\in D$
\begin{equation}
-\lambda\left(D\right)\leq\liminf_{t\to\infty}\frac{\log\mathbb{P}_{x}\left(\tau_{D}>t\right)}{t}\leq\limsup_{t\to\infty}\frac{\log\mathbb{P}_{x}\left(\tau_{D}>t\right)}{t}\leq-\frac{\lambda\left(D\right)}{1+\frac{d}{4\alpha}}.\label{eq:Main-Simplified-Asymp-1}
\end{equation}
\end{thm}

We emphasize that the result above is not a straightforward extension of the local case. The polynomial decay of a non-local heat kernel introduces significant obstacles when applying classical methods developed for local processes, where the heat kernels satisfy sub-Gaussian bounds. Our approach instead relies on a self-improvement iteration technique.
The constant $-\frac{\lambda\left(D\right)}{1+\frac{d}{4\alpha}}$ is best possible under our approach. To our knowledge, these results are the first to provide explicit asymptotic bounds of the form  \eqref{eq:Main-Simplified-Asymp-1} for non-local jump processes in general unbounded domains $D$ with $\lambda(D)>0$.

One useful application of the above result is that it provides a probabilistic interpretation of the spectral property $\lambda(D)>0$ of a domain.

\begin{corollary}
\label{cor:Equivalences-intro-general}Assuming the hypothesis of Theorem \ref{Main:Thm:StableProc} we have that for any $D\subset\mathbb{R}^{d}$,  
\begin{equation}\label{eq:lambda-pos-1}
\lambda\left(D\right)>0\iff\sup_{x\in D}\mathbb{E}_{x}\left[\tau_{D}\right]<\infty.
\end{equation}
In particular, there exists a universal constant $C>0$ such that for all domains $D \subset \mathbb{R}^d$ satisfying $\lambda (D)>0$, we have
\[
1\leq\lambda\left(D\right)\sup_{x\in D}\mathbb{E}_{x}\left[\tau_{D}\right]\leq C.
\]
\end{corollary}

We point that in the case of $\alpha$-stable jump processes, this result is well-known. See for example Theorem 1.1, equation (1.5), for the case $p=q=2$ in \cite{Franzina-2019}. We also point to the works of \cite[Theorem 3.1]{Giorgi-Smits-2010} and \cite[Theorem 3.1]{Panzo2021}. Our contribution in this paper is the generalization of Corollary \ref{cor:Equivalences-intro-general}  to metric measure spaces which is done in Corollary \ref{cor:Equivalences-gen}.

Furthermore, if the domain $D$ is assumed to be $C^{1,1}$, this equivalence can be extended to a geometric characterization of the domain via its \textbf{inradius} $R_D$, as specified in the theorem below. We say that $D$ is (uniformly) $C^{1,1}$
if there exists an $r>0$ such that for every $Q\in\partial D$,
there exist points $x_{\rm in}\in D,x_{\rm out}\in D^{c}$ with $B\left(x_{\rm in},r\right)\subset D$
and $B\left(x_{\rm out},r\right)\subset D^{c}$, tangent at $Q$, and satisfying $\overline{B\left(x_{\rm in},r\right)}\cap\overline{B\left(x_{\rm out},r\right)}=\left\{ Q\right\} $. The inradius is defined as the radius of the largest open ball contained in $D$, namely, 
\begin{equation}
R_D = \sup_{x\in D} \delta_D(x)   
\end{equation}
where $\delta_D(x):=d(x,D^c)$ denotes the distance from $x$ to $D^c$. 


In the following we prove sharp bounds for exit times in terms of $\delta_D \left(x\right)^{\alpha/2}$ for any $C^{1,1}$ domain that satisfies $\lambda (D)>0$. We are also able to show the equivalence of $\lambda (D)>0$ or $\sup_{x\in D}\mathbb{E}_x[\tau_D]<\infty$ with the condition $R_D<\infty$. We refer to the results of \cite{Bianchi-Brasco-2024,Bianchi-Brasco-2022} where universal bounds are given for $\lambda \left(D\right)$ in terms of the inradius $R_D$ for specific classes of domains. In particular \cite[Corollary 6.1]{Bianchi-Brasco-2022} proves the equivalence of $\lambda (D)>0$ with $R_D<\infty$ for simply connected domains on $\mathbb{R}^2$ when $1<\alpha<2$. We also point to the work of \cite[Theorem 4.6]{Barrier-Bogdan-Grzywny-Ryznar-2015} for sharp bounds for the mean exit time of jump processes for the case of {\it bounded} $C^{1,1}$ domains.


\begin{thm}\label{cor:Equivalences}
Suppose $D\subset\mathbb{R}^{d}$ is a uniformly $C^{1,1}$ domain
at scale $r>0$ and $0<\alpha<2$. 

(1) If $\lambda\left(D\right)>0$, 
then there exist  constants $C_{r,d},C_{D,r,d}>0$ such that for any $x\in D$, 
\begin{equation}
C_{r,d}\delta_{D}\left(x\right)^{\alpha/2}\leq\mathbb{E}_{x}\left[\tau_{D}\right]\leq C_{D,r,d}\delta_{D}\left(x\right)^{\alpha/2}.\label{eq:Mean-bound}
\end{equation} 

(2) If $R_{D}<\infty$ then 
\begin{equation}\label{eq:lambda-lower}
\lambda\left(D\right)\geq\frac{c_{d,\alpha}}{2}\frac{\omega_{d}r^{d}}{\left(R_{D}+2r\right)^{d+\alpha}},
\end{equation}
where $c_{d,\alpha}=2^{\alpha}\pi^{-d/2}\Gamma\left(\frac{d+\alpha}{2}\right)\left|\Gamma\left(-\frac{\alpha}{2}\right)\right|^{-1}$.

As a consequence, the following equivalences hold:
\begin{equation}
R_{D}<\infty\iff\lambda\left(D\right)>0\iff\sup_{x\in D}\mathbb{E}_{x}\left[\tau_{D}\right]<\infty.\label{eq:C11-equiv}
\end{equation}
\end{thm}

Another consequence of Theorem \ref{Main:Thm:StableProc} is an upper bound for the Dirichlet heat kernel of $\{X_t\}_{t\ge0}$ in unbounded domains.  There is extensive literature on such estimates for both local  (e.g., \cite{Zhang-2002,Davies-1987,Gyrya-Saloff-Coste-2011}) and non-local processes (e.g., \cite{Chen-Kim-Song-2010,Siudeja-2006,Kim-Song-2014,Chen-Kumagai-Wang-2020b}).  Our contribution addresses convex domains with $\lambda (D)>0$. For convex (possibly unbounded) domains, Siudeja \cite[Theorem 1.6]{Siudeja-2006} obtained an upper bound for the Dirichlet heat kernel of an $\alpha$-stable process.
However, this estimate does not provide an explicit exponential rate of decay as $t\to\infty$. In view of \eqref{eq:Main-Simplified-Asymp-1}, one expects this rate to be governed by $\lambda(D)$. The next proposition gives an upper bound that makes this dependence explicit.

For the rest of the paper, we write $f\left(x\right)\asymp g\left(x\right)$
if and only if there exists $c_{1},c_{2}>0$ such that $c_{1}f\left(x\right)\leq g\left(x\right)\leq c_{2}f\left(x\right)$.

\begin{prop}\label{prop-kernel-bd}
Suppose $0<\alpha<2$. Let $\epsilon\in\left(0,1\right)$. For any convex domain $D\subset\mathbb{R}^{d}$,
there exists a constant $C_{\epsilon}>0$ such that
\begin{equation}
p_{D}\left(t,x,y\right)\leq C_{\epsilon}e^{-\frac{1}{2}\frac{\left(1-\epsilon\right)}{1+\frac{d}{4\alpha}}\lambda\left(D\right)t}\left(1\wedge\frac{\delta_{D}\left(x\right)^{\alpha/2}}{\sqrt{t}}\right)\left(1\wedge\frac{\delta_{D}\left(y\right)^{\alpha/2}}{\sqrt{t}}\right)p\left(t,x,y\right),\label{eq:HK-up}
\end{equation}   
where $p(t,x,y)$ is the free heat kernel that satisfies 
\[
p\left(t,x,y\right)\asymp t^{-d/\alpha}\wedge\frac{t}{\left|x-y\right|^{d+\alpha}}.
\]
\end{prop}





In the case of non-local processes in $C^{1,1}$ domains in $\R^d$, we point to the seminal work of Chen-Kim-Song \cite[Theorem 1.1]{Chen-Kim-Song-2010}, where sharp heat kernel
bounds were established for general $C^{1,1}$ open sets $D\subset\mathbb{R}^{d}$. Specifically:\\
(1) When $D$ is a $C^{1,1}$  open (possibly unbounded) set, they proved
\begin{equation}\label{eq:Chen-Kim-Song-HK1}
p_{D}\left(t,x,y\right)\asymp_{T}\left(1\wedge\frac{\delta_{D}\left(x\right)^{\alpha/2}}{\sqrt{t}}\right)\left(1\wedge\frac{\delta_{D}\left(y\right)^{\alpha/2}}{\sqrt{t}}\right)p\left(t,x,y\right),0\leq t\leq T.
\end{equation}
(2) When $D$ is a $C^{1,1}$ bounded open set, they showed that
\begin{equation}\label{eq:Chen-Kim-Song-HK2}
p_{D}\left(t,x,y\right)\asymp_{T}e^{-\lambda_{1}\left(D\right)t}\delta_{D}\left(x\right)^{\alpha/2}\delta_{D}\left(y\right)^{\alpha/2} ,t\geq T.
\end{equation}  
For certain unbounded domains where $\lambda(D)=0$, it has been shown that \eqref{eq:Chen-Kim-Song-HK1} extends to large time $t>T$ (see \cite{Chen-Tokle-2011,Siudeja-2006}).  
However, when $\lambda(D)>0$, both expressions \eqref{eq:Chen-Kim-Song-HK1} and  \eqref{eq:Chen-Kim-Song-HK2} have limitations: they do not fully cover all $t>0$ and all unbounded domains with $\lambda(D)>0$ that have $C^{1,1}$ boundaries.
We believe that the  double-sided heat kernel estimates valid for all $t>0$ should take the form of \eqref{eq:HK-up}, up to a constant in the exponential term. However, deriving such a result is far from a straightforward combination of \eqref{eq:Chen-Kim-Song-HK1} and  \eqref{eq:Chen-Kim-Song-HK2}, and we leave it as an open direction for future research.

While it is conjectured and widely believed that \eqref{eq:coj-main} should hold for a jump process in a general domain $D\subset\R^d$ satisfying $\lambda\left(D\right)>0$, to our knowledge, \eqref{eq:Main-Simplified-Asymp-1} currently provides the best available bound. However, by restricting to certain specific class of domains,  we are able to obtain the desired matching upper and lower bounds for \eqref{eq:Main-Simplified-Asymp-1}. Specifically, we prove this for increasing horn-shaped domains as defined below.

\begin{df}\label{def-hornshape}
Let $H\subset\mathbb{R}^{d}$ be an open set in $\mathbb{R}^{+}\times\mathbb{R}^{d-1}$. For each $x_1\in \mathbb{R}^{+}$, define  the cross section of $H$ at $x_1$ by
\[
H(x_1):=\{x'\in\R^{d-1}: (x_1, x')\in H\},
\]
so that $\left\{ x_1\right\} \times H(x_1)=H\cap\left(\left\{ x_1\right\} \times\mathbb{R}^{d-1}\right)$. We say that $H$ is (increasing) \textbf{horn-shaped} if it satisfies the monotonicity condition: for any $x_1<y_1$, we have $H(x_1)\subset H(y_1)$.\\
Consider the set $h\subset \R^{d-1}$ defined by
\begin{align}\label{eq-h}
h=\bigcup_{x_1>0}H(x_1).
\end{align}
It is the projection of $H$ onto $\mathbb{R}^{d-1}$.  When $h$ is bounded,  it is known that the generator of the 
 $\left(d-1\right)$ dimensional $\alpha-$stable process
killed upon its exiting time of $h$ has a discrete spectrum. We denote by $\lambda_1(h)$ its first Dirichlet eigenvalue.

\end{df}

Horn-shaped domains have been extensively studied in the literature (e.g. \cite{Chen-Kim-Wang-2022,Siudeja-2006,Banuelos-Berg-1996,Mendez-Hernandez-2007,Cranston-Li-Horn-1997}), with much of the focus on the decreasing ones, as in \cite{Chen-Kim-Wang-2022, Banuelos-Berg-1996,Cranston-Li-Horn-1997}. In this work  we focus on increasing horn-shaped domains as done by Siudeja in \cite{Siudeja-2006} (see Figure \ref{fig:horn-shape} for an example of an increasing horn-shaped domain).  Mendez-Hernandez \cite{Mendez-Hernandez-2007} also studied increasing horn-shaped domains, specifically in the case where $h=\mathbb{R}^{d-1}$ and $\lambda(H)=0$, and obtained  the exact large time asymptotics for $\mathbb{P}_x(\tau_D>t)$ for various examples.
In contrast, the result presented in the theorem below  studies increasing horn-shaped domains with $\lambda(H)>0$. Additionally, we do not assume convexity, distinguishing our approach from those in \cite{Siudeja-2006, Mendez-Hernandez-2007}. 



\begin{thm}
\label{thm:Horn-Domains-1}Let $X_{t}$ be a symmetric $\alpha-$stable
process on $\mathbb{R}^{d}$ with $0<\alpha<2$. Let $H$ be a horn-shaped domain as in Definition \ref{def-hornshape} with a bounded projection $h$ as given in \eqref{eq-h}. We denote by $\tau_H$ the exit time of $X_t$ and by $\lambda(H)$ its bottom of the spectrum. Then for any starting point $x\in H$, we have
\begin{equation}\label{eq:Horn-sharp}
\lim_{t\to\infty}\frac{\log\mathbb{P}_{x}\left(\tau_{H}>t\right)}{t}=-\lambda\left(H\right)=-\lambda_{1}\left(h\right).
\end{equation}
\end{thm}

Before ending the introduction, we point out that Theorem \ref{Main:Thm:StableProc} is in fact derived in a more general metric measure space setting (see Theorem \ref{thm:MainTheorem1}). There has been significant progress on addressing \eqref{eq:coj-main} for local processes in geometric settings beyond $\R^d$. For instance, the results in \cite{Carfagnini-Gordina-Teplyaev-2024,Carfagnini-Gordina-2024} imply \eqref{eq:coj-main} for diffusion processes in all bounded domains within a general metric measure space and a sub-elliptic setting. In the unbounded case, \cite{Mariano-Wang-2022} proves \eqref{eq:coj-main} for all domains with positive bottom of spectrum in a metric measure space setting. The main contribution of this paper, Theorem \ref{Main:Thm:StableProc}, provides both upper and lower bound for the survival probability  $\mathbb{P}_{x}\left(\tau_{D}>t\right)$ of all jump processes in unbounded domains with $\lambda(D)>0$, within general metric measure spaces. This framework includes, for instance, jump processes in sub-Riemannian manifolds, fractals, and others non-Euclidean spaces. See \cite[Section 8]{Mariano-Wang-2022} for a range of examples of metric measure spaces that satisfy the geometric hypothesis. See also \cite[Chapter 6]{Chen-Kumagai-Wang-2021-MAIN} for many examples of processes satisfying our conditions.

We organize the paper as follows. In Section \ref{sec2} we introduce the  metric measure space setting along with the main assumptions, and we state the main results in this framework. These include the sharp asymptotic upper and lower bounds given in  Theorems  \ref{thm:MainTheorem1} and \ref{Thm2}, which in particular cover Theorem \ref{Main:Thm:StableProc}. Sections \ref{sec3} and \ref{sec4} are devoted to the proofs of Theorems \ref{thm:MainTheorem1} and \ref{Thm2}, respectively. In Section \ref{sec5}, we give the proofs of Corollary \ref{cor:Equivalences-intro-general}, Theorem \ref{cor:Equivalences}, Proposition \ref{prop-kernel-bd}, and Theorem \ref{thm:Horn-Domains-1} concerning horn-shaped domains. Finally, in Section \ref{sec:ex}, we present non-trivial examples of unbounded domains to which our results apply.

\section{Main Results for Jump Processes in Metric Measure Spaces}\label{sec2}

Let $\left(M,d,\mu\right)$ be a metric measure space endowed with a regular {Dirichlet form} $\left(\mathcal{E},\mathcal{F}\right)$
on $L^{2}\left(M,\mu\right)$. It is known that there exists a {Hunt process} $\left\{ X_{t}\right\} _{t\geq0}$ associated with it (see \cite{Fukushima-book-ed1-1994}). In this paper, we focus on pure-jump stochastic processes, which are generated by non-local Dirichlet forms of the form: for any $f,g\in\mathcal{F}$,
\begin{align}\label{eq-Diri}
    \mathcal{E}\left(f,g\right)=\int_{M\times M\backslash\triangle}\left(f\left(x\right)-f\left(y\right)\right)\left(g\left(x\right)-g\left(y\right)\right)J\left(dx,dy\right),
\end{align}
where $\triangle$ is the diagonal set $\left\{ \left(x,x\right):x\in M\right\}$, and $J\left(\cdot,\cdot\right)$ is a symmetric Radon measure on $M\times M\backslash\triangle$.

Consider the heat {semigroup} $\left\{ P_{t}\right\} _{t\geq0}$ that is associated with $\left(\mathcal{E},\mathcal{F}\right)$. It is known that, outside of a properly exceptional set $\mathcal N$, the following holds for any bounded Borel measurable function $f$ on $M$ that (see \cite{Chen-Kumagai-Wang-2021-MAIN}): 
\[
P_{t}f\left(x\right)=\mathbb{E}_{x}\left[f\left(X_{t}\right)\right], \quad x\in M\setminus \mathcal N.
\]
Recall that a set $\mathcal{N}$ is \textbf{properly exceptional} for $X$ if $\mathcal{N}$ is nearly Borel, $\mu\left(\mathcal{N}\right)=0$  and
\[\mathbb{P}_{x}\left(X_{t} \in\mathcal{N}\text{ or } X_{t-}\in\mathcal{N}\text{ for some }t\geq 0\right)=0,
\]
for all $x\in M\backslash\mathcal{N}$, and where we take $X_{0-}:=X_0$.  
Throughout this paper, many of our discussions and conclusions are made outside of a properly exceptional set, which is not explicitly specified. For simplicity of notation, unless otherwise stated, we will fix a process $\{X_{t}\}_{t\geq 0}$ and a set $\mathcal N$ and denote $M_0=M\backslash\mathcal{N}$.

The {heat kernel} $p\left(t,x,y\right)$, $x,y\in M$, $t\ge0$ of the heat semigroup $\left\{ P_{t}\right\} _{t\geq0}$, if exists, is a non-negative, measurable function on $\R_{\ge0}\times M_0\times M_0$ such that for any $f\in L^{\infty}\left(M,\mu\right)$ we have 
\begin{equation}\label{eq:heat-kernel-def}
P_{t}f\left(x\right)=\int_{M}f\left(y\right)p\left(t,x,y\right)d\mu\left(y\right)
\end{equation}
for all $x\in M_0$. It satisfies properties
such as Markovian, symmetry, Chapman-Kolmogorov and approximation
of identity properties. (See \cite{Chen-Kumagai-Wang-2021-MAIN} for more details). 

\begin{assumption}\label{assump}
Let $\left(M,d,\mathcal{E},\mathcal{F},\mu\right)$ be a Dirichlet metric measure space, and let $\left\{ X_{t}\right\} _{t\geq0}$ be a {Hunt process}  associated with it. We impose the following assumptions:
\begin{itemize}
    \item [(a)] (\textbf{Volume doubling property} $(VD)$). There exists a constant
$C_{\mu}\geq1$ such that for all $x\in M$ and $r>0$, $V\left(x,2r\right)\leq C_{\mu}V\left(x,r\right)$.
This condition is equivalent to the existence of
a (possibly different) $C_{\mu}\geq1$ and $d_{1}>0$ such that for
all $x\in M$ and $0<r\leq R$, 
\begin{equation}\label{eq:VD-regularity}
\frac{V\left(x,R\right)}{V\left(x,r\right)}\leq C_{\mu}\left(\frac{R}{r}\right)^{d_{1}}.
\end{equation}

\item [(b)] (\textbf{Reverse volume doubling} $(RVD)$). There exists constants $d_{2},c_{\mu}>0$ such that
for all $x\in M$ and $0<r\leq R$, 
\[
c_{\mu}\left(\frac{R}{r}\right)^{d_{2}}\leq\frac{V\left(x,R\right)}{V\left(x,r\right)}.
\]

\item [(c)]  (\textbf{Heat kernel bound} \textbf{condition} $HK\left(\phi\right)$). The process $\{X_{t}\}_{t\geq0}$ admits a {heat kernel} $p\left(t,x,y\right):M_0\times M_0\to\left(0,\infty\right)$, $M_0:=M\setminus \mathcal N$ where $\mathcal N$ is a properly exceptional set, possibly larger than the exceptional set introduced above. We assume that $p(t,x,y)$ satisfies the Markovian, symmetric and Chapman-Kolmogorov properties on $M_0$, and that there exist $c_1, c_2>0$ such that for
all $t>0$ and all $x,y\in M_0$, we have
\begin{align}\label{eq:HK}
 & c_{1}\left(\frac{1}{V\left(x,\phi^{-1}\left(t\right)\right)}\wedge\frac{t}{V\left(x,d\left(x,y\right)\right)\phi\left(d\left(x,y\right)\right)}\right)\nonumber \\
 & \leq p\left(t,x,y\right)\leq c_{2}\left(\frac{1}{V\left(x,\phi^{-1}\left(t\right)\right)}\wedge\frac{t}{V\left(x,d\left(x,y\right)\right)\phi\left(d\left(x,y\right)\right)}\right),
\end{align} 
where $\phi:\mathbb{R}_{+}\to\mathbb{R}_{+}$
is a strictly increasing continuous function satisfying $\phi\left(0\right)=0$,
 $\phi\left(1\right)=1$ and  
\begin{equation}\label{eq:Phi-regularity}
c\left(\frac{R}{r}\right)^{\alpha_{1}}\leq\frac{\phi\left(R\right)}{\phi\left(r\right)}\leq c'\left(\frac{R}{r}\right)^{\alpha_{2}},\,\,\,\,\,\,\,\,\text{for all }0<r\leq R
\end{equation}
for some $c,c'>0$ and $\alpha_{2}\geq\alpha_{1}$. In this case the inverse function $\phi^{-1}$ satisfies
\begin{equation}\label{eq:Phi-regularity3}
c_{3}\left(\frac{T}{t}\right)^{1/\alpha_{2}}\leq\frac{\phi^{-1}\left(T\right)}{\phi^{-1}\left(t\right)}\leq c_{4}\left(\frac{T}{t}\right)^{1/\alpha_{1}},\,\,\,\,\text{for all }0<t\leq T.
\end{equation}
\end{itemize}
\end{assumption}

\begin{rem}
   It is enough to assume that Assumption \ref{assump} (c) holds for each \(t>0\) and
\(\mu\otimes\mu\)-a.e. \((x,y)\in M\times M\). Indeed, by
\cite[Theorem 3.1]{Barlow-Bass-Chen-Kassmann-2009}, after enlarging the
properly exceptional set if necessary, one may choose a version of the heat
kernel on \(M_0=M\setminus\mathcal N\) such that the semigroup identity,
symmetry, Chapman--Kolmogorov identity, and the estimate \eqref{eq:HK} hold
for all \(t>0\) and all \(x,y\in M_0\). 
\end{rem}

\begin{rem}
Note that under Assumption \ref{assump}, by \cite[Theorem 1.13]{Chen-Kumagai-Wang-2021-MAIN} the heat kernel condition $HK\left(\phi\right)$ implies the existence of a
non-negative symmetric function $J\left(x,y\right)$ such that for $\mu\times\mu$
a.e. $x,y\in M$ 
\[
J\left(dx,dy\right)=J\left(x,y\right)\mu\left(dx\right)\mu\left(dy\right),
\]
and additionally the existence of constants $c_5,c_6>0$ such that 
\[
\frac{c_{5}}{V\left(x,d\left(x,y\right)\right)\phi\left(d\left(x,y\right)\right)}\leq J\left(x,y\right)\leq\frac{c_{6}}{V\left(x,d\left(x,y\right)\right)\phi\left(d\left(x,y\right)\right)}.
\]   
\end{rem}

This set of assumptions is standard for studying jump processes on general metric measure spaces, see \cite{Chen-Kumagai-Wang-2021-MAIN} for instance. It covers a wide range of jump processes in many interesting metric measure settings. See \cite[Chapter 6]{Chen-Kumagai-Wang-2021-MAIN} for examples of various subordinated processes of the form $X_t=Z_{S_t}$ where $S_{t}$, $t\ge0$ is an $\frac{\alpha}{2}$-stable subordinator independent of the diffusion $Z_t$.

An important and classical example is the $\alpha$-stable process on $\R^d$.

\begin{example}
[Symmetric $\alpha$-stable processes on $\mathbb{R}^{d}$]\label{ex-stable-proc}
Let $(X_t)_{t\ge0}$ be an $\alpha$-stable process in $\mathbb{R}^{d}$ with $0<\alpha<2$. This process  satisfies the condition $HK\left(\phi\right)$ with $\phi\left(r\right)=r^{\alpha}$ and $(VD)$ with $d_{1}=d_{2}=d$
since $V\left(x,r\right)=\omega_{d}r^{d}$. 
\end{example}





For any non-empty open set $D\subset M$ , we define $C_{0}\left(D\right)$
to be the space of continuous functions with compact support in $D$, equipped with the sup-norm. We define $\mathcal{F}\left(D\right)$ as the closure of $\mathcal{F}\cap C_{0}\left(D\right)$ in the norm of $\mathcal{F}.$ Let $\mathcal{L}_{D}$ be the generator of $\left(\mathcal{E},\mathcal{F}\left(D\right)\right)$
and let $P_{t}^{D}$ be its associated heat semigroup. The \textbf{bottom
of the spectrum} of $\mathcal{L}_{D}$ is defined by 
\[
\lambda\left(D\right):=\inf\text{spec }\left(\mathcal{L}_{D}\right)=\inf_{u\in\mathcal{F}\left(D\right)\backslash\left\{ 0\right\} }\frac{\mathcal{E}\left(u,u\right)}{\left\Vert u\right\Vert _{2}^{2}}.
\]
For the associated Hunt process $(X_t)_{t\ge0}$, we define the \textbf{first exit time} from $D$ as 
\[
\tau_{D}:=\inf\left\{ t>0:X_{t}\notin D\right\}.
\]
We now present our main results on the survival probability of  symmetric jump-processes.
\begin{thm}[Spectral bounds for general domains]
\label{thm:MainTheorem1} Assume
the metric measure space $\left(M,d,\mu\right)$ is endowed with a
regular Dirichlet form $\left(\mathcal{E},\mathcal{F}\right)$, and
satisfies Assumption \ref{assump}. Let $D\subset M$ be an open set satisfying $\lambda\left(D\right)>0$. Then for every $\epsilon\in\left(0,1\right)$
 there exists a constant $C_{\epsilon}>0$ such that 
\begin{align}
\sup_{x\in D\setminus\mathcal{N}}\mathbb{P}_{x}\left(\tau_{D}>t\right) & \leq C_{\epsilon}\exp\left(-\frac{1-\epsilon}{1+\frac{d_1}{4\alpha_1}}\lambda\left(D\right)t\right),\label{eq:MainTheorem1-1}\\
\esup_{x\in D}\mathbb{P}_{x}\left(\tau_{D}>t\right) & \geq e^{-\lambda(D)t},\label{eq:MainTheorem1-1b}
\end{align}
for all $t>0$. Here $\mathcal{N}$ is a properly exceptional set. The constants $d_{1},\alpha_{1}$ are given in $\left(\ref{eq:VD-regularity}\right)$ and $\left(\ref{eq:Phi-regularity}\right)$. 
\end{thm}

Using Theorem \ref{thm:MainTheorem1} along with a point-wise asymptotic lower bound we obtain the following
\begin{thm}\label{Thm2}
Assume the Dirichlet metric measure space $\left(M,d,\left(\mathcal{E},\mathcal{F}\right), \mu\right)$ satisfies Assumption \ref{assump}. For any open set $D\subset M$ we have
\begin{equation}\label{eq:current-best-asymp}
\limsup_{t\to\infty}\frac{\log\mathbb{P}_{x}\left(\tau_{D}>t\right)}{t}\leq-\frac{\lambda\left(D\right)}{1+\frac{d_{1}}{4\alpha_{1}}}
\end{equation}
for any $x\in D\backslash \mathcal{N}$. 
Suppose further that $D\subset M$ is a domain such that the heat
kernel $p_{D}$ of $\left(\mathcal{E},\mathcal{F}\left(D\right)\right)$ exists, is continuous $\left(0,\infty\right)\times D\times D$ and strictly positive on $D\times D$. If
$\lambda\left(D\right)>0$, then 
\begin{equation}\label{eq:MainTheorem1-2}
-\lambda\left(D\right)\leq\liminf_{t\to\infty}\frac{\log\mathbb{P}_{x}\left(\tau_{D}>t\right)}{t} \quad\text{for all $x\in D$.}
\end{equation}
 \end{thm}

To our knowledge,  the results in \eqref{eq:MainTheorem1-1} and \eqref{eq:current-best-asymp} provide the best known upper bounds for the asymptotic $\mathbb{P}_{x}\left(\tau_{D}>t\right)$ in terms of $\lambda (D)$ for general unbounded domains.  These bounds are optimal within the framework of our approach, which is based on a self-improving iteration technique that departs significantly from methods used in the local case. 

As a nice consequence, we obtain the following equivalence, which can be used to characterize a positive bottom of spectrum. For related results in the local setting,  see \cite{Banuelos-Carrol1994,Giorgi-Smits-2010,Vandenberg-Carroll-2009,Banuelos-Mariano-Wang-2024,Mariano-Wang-2022}. 
\begin{corollary}
\label{cor:Equivalences-gen}Assuming the hypothesis of Theorem \ref{thm:MainTheorem1},
for any $D\subset M$, we have 
\begin{equation}
\lambda\left(D\right)>0\iff\sup _{x\in D\setminus \mathcal{N}}\mathbb{E}_{x}\left[\tau_{D}\right]<\infty.\label{eq:lambda-pos-1-gen}
\end{equation}
In particular, there exists a universal constant $C>0$ such that for all domains $D \subset M$ satisfying $\lambda (D)>0$, we have
\[
1\leq\lambda\left(D\right)\sup _{x\in D\setminus \mathcal{N}}\mathbb{E}_{x}\left[\tau_{D}\right]\leq C.
\]
\end{corollary}

\section{Proof of Theorem \ref{thm:MainTheorem1}}\label{sec3}

First, we address the existence of the Dirichlet heat kernel $p_D$.

Suppose we assume all of Assumption \ref{assump}. By \cite[Prop. 7.6]{Chen-Kumagai-Wang-2021-MAIN} we have that $(VD),$ $ (RVD)$, $(\ref{eq:Phi-regularity})$ and  $HK\left(\phi\right)$ implies a 
 Faber-Krahn inequality. By \cite[Prop. 7.7]{Chen-Kumagai-Wang-2021-MAIN}, under $(VD)$, $(\ref{eq:Phi-regularity})$ and a Faber-Krahn inequality implies that there exists a properly exceptional set $\mathcal{N}\subset M$
such that, for any open set $D\subset M$, the semigroup $\left\{ P_{t}^{D}\right\} $
admits a heat kernel $p_{D}\left(t, x,y\right)$ defined on $D\backslash\mathcal{N}\times D\backslash\mathcal{N}$.
We extend $p_{D}$ to $D\times D$ by defining $p_{D}=0$ on $\left(D\times D\right)\backslash\left(D\backslash\mathcal{N}\times D\backslash\mathcal{N}\right)$. Alternatively, one can use \cite[Theorem 3.1]{Barlow-Bass-Chen-Kassmann-2009} to prove the existence of the heat kernel $p_D(t,x,y)$.

Next, we observe that the lower bound $(\ref{eq:MainTheorem1-1b})$ follows directly from  \cite[Prop.
3.6]{Mariano-Wang-2022}. The rest of this section is  then dedicated to proving the upper bound $(\ref{eq:MainTheorem1-1})$, which we divide into two subsections. In Section \ref{sec-prel}, we establish a preliminary upper bound estimate on the survival probability. This estimate will then be refined in Section \ref{sec-refine} using an iteration technique.

We note that in the proof below the constant $C$ may change from line to line. 

\subsection{A Preliminary Estimate}\label{sec-prel}

The Proposition below establishes an initial upper bound for the survival probability. This estimate serves as a starting point and will be improved in the next subsection.

\begin{prop}\label{prop:Prop-Main-Est-1}
Under the same assumptions as in Theorem \ref{thm:MainTheorem1}, there exists a constant $C>0$
such that for every $\epsilon\in\left(0,1\right)$ and every open set
$D\subset M$ with $\lambda\left(D\right)>0$, we have that
\begin{equation}\label{eq:Prop-Main-Est-1}
\sup_{x\in D\setminus \mathcal{N}}\mathbb{P}_{x}\left(\tau_{D}>t\right)\leq C\left(1+\frac{1}{\epsilon}\right)\exp\left(-\frac{1-\epsilon}{1+\frac{d_{1}}{2\alpha_{1}}}\lambda\left(D\right)t\right).
\end{equation}
where $\mathcal{N}$ is a properly exceptional set.
\end{prop}

\begin{proof}
Fix an $R>0$ (to be chosen later). For any $x\in D\backslash\mathcal{N}$, we split the probability as follows,
\begin{align}\label{eq:BoundTheorem-0}
\mathbb{P}_{x}\left(\tau_{D}>t\right) & =\mathbb{P}_{x}\left(\tau_{D}>t,X_{t}\in B\left(x,R\right)\right)+\mathbb{P}_{x}\left(\tau_{D}>t,X_{t}\notin B\left(x,R\right)\right)\nonumber \\
 & =:I+II.
\end{align}
\textbf{Step 1:} Estimate on $I$. 
Recall the heat semigroup $P_{t}^{D}$ associated to $(\mathcal{E},\mathcal{F}(D))$. Outside of a properly exceptional set $\mathcal{N}$ it satisfies  $P_{t}^{D}f(x)=\mathbb{E}_{x}\left[f\left(X_{t}\right),\tau_{D}>t\right]$ for all $f\in \mathcal{F}(D)$ and for all $x\in M\setminus \mathcal{N}$. Let $p_{D}\left(t, x,y\right)$ be the associated heat kernel as previously specified. We can then write
\begin{align}
I & =\int_{M}1_{D\cap B\left(x,R\right)}\left(y\right)p_{D}\left(t,x,y\right)d\mu\left(y\right)
\end{align}
Applying Chapman-Kolmogorov identity and Fubini's theorem we have
\begin{align}\label{eq-I-est-1}
    I 
 & =\int_{M}1_{D\cap B\left(x,R\right)}\left(y\right)\int_{D}p_{D}\left(\epsilon t,x,z\right)p_{D}\left(\left(1-\epsilon\right)t,z,y\right)d\mu(z)d\mu\left(y\right)\nonumber \\
 & =\int_{D}p_{D}\left(\epsilon t,x,z\right)P_{\left(1-\epsilon\right)t}^{D}1_{D\cap B\left(x,R\right)}\left(z\right)d\mu(z). 
\end{align}
By the spectral theorem we know that the  semigroup operator $P_{\left(1-\epsilon\right)t}^{D}$ has the $L^2$-operator norm bound:
\begin{equation}\label{eq:thm:HK-simplest-L2spectral-1}
\left\Vert P_{\left(1-\epsilon\right)t}^{D} 1_{D\cap B\left(x,R\right)} \right\Vert _{2}\leq e^{-\lambda\left(D\right)\left(1-\epsilon\right)t}\left\Vert 1_{D\cap B\left(x,R\right)}\right\Vert _{2}.
\end{equation}
Using the Cauchy--Schwarz inequality on \eqref{eq-I-est-1} and plugging in \eqref{eq:thm:HK-simplest-L2spectral-1} we obtain for any $\epsilon\in\left(0,1\right)$ that
\begin{align}\label{eq:BoundTheorem-1}
I  & \leq\left(\int_{D}p_{D}\left(\epsilon t,x,z\right)^{2}d\mu(z)\right)^{1/2}e^{-\lambda\left(D\right)\left(1-\epsilon\right)t}\left\Vert 1_{D\cap B\left(x,R\right)}\right\Vert _{2}.
\end{align}
Moreover, note that
\begin{align*}
    \int_{D}p_{D}\left(\epsilon t,x,z\right)^{2}d\mu(z)
    &\le \esup_{z\in D}p_{D}\left(\epsilon t,x,z\right)\int_{D}p_{D}\left(\epsilon t,x,z\right)d\mu(z)
     \le \frac{C}{V\left(x,\phi^{-1}\left(2\epsilon t\right)\right)}
\end{align*}
where the last inequality follows from the fact that $p_D(t,x,z)\le p_M(t, x,z)$ for all $t\ge0$, $x,z\in D$, and that $p_M$ satisfies the $HK(\phi)$ condition. Therefore  \eqref{eq:BoundTheorem-1} becomes
\begin{align*}
I 
 \leq 
 C\left(\frac{V\left(x,R\right)}{V\left(x,\phi^{-1}\left(2\epsilon t\right)\right)}\right)^{1/2}e^{-\lambda\left(D\right)\left(1-\epsilon\right)t}.
\end{align*}
Let $R\geq \phi^{-1}\left(2\epsilon t\right)$. Applying $(VD)$ we then obtain 
\begin{align}\label{eq:Theorem-proof-I-bound}
I & \leq C\left(1+\left(\frac{R}{\phi^{-1}\left(2\epsilon t\right)}\right)^{d_{1}/2}\right)e^{-\lambda\left(D\right)\left(1-\epsilon\right)t},
\end{align}
where $d_1>0$ is the constant specified as in \eqref{eq:VD-regularity}.
\\
\noindent
\textbf{Step 2:} Estimate on $II$. From \eqref{eq:BoundTheorem-0} we can rewrite II as
\begin{align}
II & =\int_{D\cap B\left(x,R\right)^{c}}p_{D}\left(t,x,y\right)d\mu\left(y\right).\nonumber
\end{align}
Using the fact that $p_{D}\left(t,x,y\right)$ satisfied $HK(\phi)$ in \eqref{eq:HK} we obtain that 
\begin{align}\label{eq:Theorem-proof-II-bound-mid1}
II 
 & \leq c_{2}t\int_{B\left(x,R\right)^{c}}\frac{1}{V\left(x,d\left(x,y\right)\right)\phi\left(d\left(x,y\right)\right)}d\mu\left(y\right)
\end{align}
Moreover, since $\phi$ satisfies $(\ref{eq:Phi-regularity})$, it follows from \cite[Lemma 2.1]{Chen-Kumagai-Wang-2021-MAIN} that 
there exists a constant $c_{0}>0$ such that for any $x\in M$ and $R>0$,
\[
\int_{B\left(x,R\right)^{c}}\frac{1}{V\left(x,d\left(x,y\right)\right)\phi\left(d\left(x,y\right)\right)}d\mu\left(y\right)\leq\frac{c_{0}}{\phi\left(R\right)}
\]
Plugging this into \eqref{eq:Theorem-proof-II-bound-mid1} we obtain for some $C>0$ that 
\begin{align}\label{eq:Theorem-proof-II-bound}
II  \leq 
 \frac{Ct}{\phi\left(R\right)}.
\end{align}
\textbf{Step 3:} Optimization of bounds for $I$ and $II$. Combining $(\ref{eq:Theorem-proof-I-bound})$
and $(\ref{eq:Theorem-proof-II-bound})$ we have that

\begin{align}\label{eq:BoundTheorem-3}
\mathbb{P}_{x}\left(\tau_{D}>t\right) & \leq C\left(1+\left(\frac{R}{\phi^{-1}\left(2\epsilon t\right)}\right)^{d_{1}/2}\right)e^{-\lambda\left(D\right)\left(1-\epsilon\right)t}+\frac{Ct}{\phi\left(R\right)}.
\end{align}

We then set $R=\phi^{-1}\left(2\epsilon te^{At\lambda(D)}\right)>\phi^{-1}\left(2\epsilon t\right)$, where $A>0$ is a constant to be determined later. The above estimate  becomes
\begin{align*}
\mathbb{P}_{x}\left(\tau_{D}>t\right) & \leq C\left(1+\left(\frac{\phi^{-1}\left(2\epsilon te^{At\lambda(D)}\right)}{\phi^{-1}\left(2\epsilon t\right)}\right)^{d_{1}/2}\right)e^{-\lambda\left(D\right)\left(1-\epsilon\right)t}+\frac{C}{2\epsilon}e^{-A\lambda(D)t}.
\end{align*}
Since $2\epsilon t\leq2\epsilon te^{At\lambda(D)}$, resorting to $(\ref{eq:Phi-regularity3})$ we obtain that
\begin{align}\label{eq:BoundTheorem-4}
\mathbb{P}_{x}\left(\tau_{D}>t\right) & \leq 
 C\left(1+e^{\frac{Ad_{1}}{2\alpha_{1}}\lambda(D)t}\right)e^{-\lambda\left(D\right)\left(1-\epsilon\right)t}+\frac{C}{2\epsilon}e^{-A\lambda(D)t}.
\end{align}
Now choose $A$ such that $\frac{Ad_{1}}{2\alpha_{1}}-\left(1-\epsilon\right)=-A$, which yields $A=\frac{1-\epsilon}{1+\frac{d_{1}}{2\alpha_{1}}}$. Plugging this into $(\ref{eq:BoundTheorem-4})$ we obtain
\begin{align*}
\mathbb{P}_{x}\left(\tau_{D}>t\right) & \leq Ce^{-\frac{1-\epsilon}{1+\frac{d_{1}}{2\alpha_{1}}}\lambda(D)t}+\frac{C}{2\epsilon}e^{-\frac{1-\epsilon}{1+\frac{d_{1}}{2\alpha_{1}}}\lambda(D)t}\\
 & \leq C\left(1+\frac{1}{\epsilon}\right)e^{-\frac{1-\epsilon}{1+\frac{d_{1}}{2\alpha_{1}}}\lambda(D)t}
\end{align*}
as desired. 
\end{proof}
\subsection{Proof of the upper bound \eqref{eq:MainTheorem1-1}}\label{sec-refine}
In this section, we refine the upper bound obtained in Proposition \ref{prop:Prop-Main-Est-1} by applying an iteration technique.
We begin with a bound on the Dirichlet heat kernel $p_{D}\left(t,x,y\right)$
in terms of the free heat kernel $p\left(t,x,y\right)$ and the survival
probability $\sup_{z\in D}\mathbb{P}_{z}\left(\tau_{D}>\frac{t}{2}\right)$, as stated in the lemma below. This is a refinement of \cite[Lemma 3.4]{Siudeja-2006}. 
\begin{lemma}\label{lem:HK-Survival-bound}
Assume the Dirichlet metric measure space $\left(M,d,\left(\mathcal{E},\mathcal{F}\right), \mu\right)$ satisfies Assumption \ref{assump}, and let $D\subset M$ be an open set. Then there exists a constant $C>0$, independent of $x,y,t$ and $D$ such that 
\begin{equation}\label{eq:HK-Survival-bound}
p_{D}\left(t,x,y\right)\leq C\sup_{z\in D}\mathbb{P}_{z}\left(\tau_{D}>\frac{t}{2}\right)p\left(t,x,y\right).
\end{equation}
for any $x,y\in D$ and $t>0$.
\end{lemma}

\begin{proof}
First, observe from the triangle inequality that for any $x,y,z\in M$, we have
\begin{equation}
\max\left\{ d\left(x,z\right),d\left(z,y\right)\right\} \geq\frac{1}{2}d\left(x,y\right).\label{eq:HK-Survival-bound-1}
\end{equation}
We can then split the domain $D$ as $D=A_{1}\left(x,y\right)\cup A_{2}\left(x,y\right)$
where
\begin{align}\label{eq-A1-A2}
A_{1}\left(x,y\right)  :=\left\{ z\in D:\frac{1}{2}d\left(x,y\right)\leq d\left(x,z\right)\right\}, \quad
A_{2}\left(x,y\right)  :=\left\{ z\in D:\frac{1}{2}d\left(x,y\right)\leq d\left(z,y\right)\right\}.
\end{align}
Using the Chapman-Kolmogorov and the splitting of $D$ above we have 
\begin{align}\label{eq-pD-I-II}
p_{D}\left(t,x,y\right) & =\int_{D}p_{D}\left(\frac{t}{2},x,z\right)p_{D}\left(\frac{t}{2},z,y\right)dz \notag\\
 & \leq\int_{A_{1}\left(x,y\right)}p_{D}\left(\frac{t}{2},x,z\right)p_{D}\left(\frac{t}{2},z,y\right)dz+\int_{A_{2}\left(x,y\right)}p\left(\frac{t}{2},x,z\right)p_{D}\left(\frac{t}{2},z,y\right)dz \notag\\
 & =:I+II.
\end{align}
We first give an estimate for $I$. For any $z\in A_{1}\left(x,y\right)$, using the fact  that $p_{D}\left(\frac{t}{2},x,z\right) \leq p\left(\frac{t}{2},x,z\right)$ and that $p\left(\frac{t}{2},x,z\right)$ satisfies $HK\left(\phi\right)$
in $\left(\ref{eq:HK}\right)$ we get 
\begin{align}\label{eq-pD-mid1}
p_{D}\left(\frac{t}{2},x,z\right) & \leq 
C_{2}\left(\frac{1}{V\left(x,\phi^{-1}\left(\frac{t}{2}\right)\right)}\wedge\frac{\frac{t}{2}}{V\left(x,d\left(x,z\right)\right)\phi\left(d\left(x,z\right)\right)}\right).
\end{align}
Since $z\in A_{1}\left(x,y\right)$ implies that $\frac{1}{2}d\left(x,y\right)\leq d\left(x,z\right)$, also using the fact that  $\phi$ is strictly increasing we have
\[
\frac{t}{V\left(x,d\left(x,z\right)\right)\phi\left(d\left(x,z\right)\right)}
\le
\frac{t}{V\left(x,\frac{1}{2}d\left(x,y\right)\right)\phi\left(\frac{1}{2}d\left(x,y\right)\right)}.
\]
Moreover, using the volume doubling condition ($VD$) in \eqref{eq:VD-regularity} we know that 
\[
\frac{V\left(x,d\left(x,y\right)\right)}{V\left(x,\frac{1}{2}d\left(x,y\right)\right)}\leq C_{\mu}2^{d_{1}}\text{ and }\frac{\phi\left(d\left(x,y\right)\right)}{\phi\left(\frac{1}{2}d\left(x,y\right)\right)}\leq c_{2}2^{\alpha_{2}}.
\]
Plugging these into the right-hand side of the above inequality we obtain
\begin{align}\label{eq-V-mid1}
    \frac{t}{V\left(x,d\left(x,z\right)\right)\phi\left(d\left(x,z\right)\right)}\leq
\frac{Ct}{V\left(x,d\left(x,y\right)\right)\phi\left(d\left(x,y\right)\right)}
\end{align}
Next, since $\phi^{-1}$ is increasing, using \eqref{eq:Phi-regularity3} we have that
\[
\frac{V\left(x,\phi^{-1}\left(t\right)\right)}{V\left(x,\phi^{-1}\left(\frac{t}{2}\right)\right)}\leq C_{\mu}\left(\frac{\phi^{-1}\left(t\right)}{\phi^{-1}\left(\frac{t}{2}\right)}\right)^{d_{1}}\leq C_{\mu}\left(c_{2}2^{1/\alpha_{1}}\right)^{d_{1}},
\]
which yields 
\begin{align}\label{eq-V-mid2}
    \frac{1}{V\left(x,\phi^{-1}\left(\frac{t}{2}\right)\right)}
    \le
    \frac{C}{V\left(x,\phi^{-1}\left(t\right)\right)}
\end{align}
for some constant $C>0$. Combining \eqref{eq-V-mid1} and \eqref{eq-V-mid2} with \eqref{eq-pD-mid1} we conclude that for any $z\in A_{1}\left(x,y\right)$ and $t>0$,
\begin{align}
p_{D}\left(\frac{t}{2},x,z\right) 
 & \leq Cp\left(t,x,y\right),\label{eq:HK-Survival-bound-2}
\end{align}
for some $C>0$. Substituting this into the expression for $I$ in \eqref{eq-pD-I-II} we get
\begin{align}\label{eq-I-bd}
I & \leq Cp\left(t,x,y\right)\int_{A_{1}}p_{D}\left(\frac{t}{2},z,y\right)dz\leq Cp\left(t,x,y\right)\sup_{y\in D\setminus  \mathcal{N}}\mathbb{P}_{y}\left(\tau_{D}>\frac{t}{2}\right).
\end{align}
Note that the sets $A_1$ and $A_2$ defined in \eqref{eq-A1-A2} are symmetric in $x$ and $y$. Thus the  same argument applies to $II$, leading to the same bound as in \eqref{eq-I-bd}, which completes the proof. 
\end{proof}

We are now ready to prove the upper bound for survival probability. 
\begin{proof}[Proof of \eqref{eq:MainTheorem1-1}]
To improve the estimate obtained in Proposition \ref{prop:Prop-Main-Est-1}, we apply an iteration technique based on its proof.  
Let $R_{2}>0$ be a parameter to be chosen later. For any $x\in D$, we split the probability $\mathbb{P}_{x}\left(\tau_{D}>t\right)$ as before:
\begin{align*}
 & \mathbb{P}_{x}\left(\tau_{D}>t\right)\\
 =& \mathbb{P}_{x}\left(\tau_{D}>t,X_{t}\in B\left(x,R_{2}\right)\right)+\mathbb{P}_{x}\left(\tau_{D}>t,X_{t}\notin B\left(x,R_{2}\right)\right)=:I_{2}+II_{2}.
\end{align*}
By the same argument as in the Step 1 of the proof of  Proposition \ref{prop:Prop-Main-Est-1}, we obtain 
\begin{equation}\label{eq:Main-Improved-1}
I_{2}\leq C\left(1+\left(\frac{R_{2}}{\phi^{-1}\left(2\epsilon t\right)}\right)^{d_{1}/2}\right)e^{-\lambda\left(D\right)\left(1-\epsilon\right)t}.
\end{equation}
The improvement is made on the bound of $II_{2}$. By invoking the heat kernel bound obtained in Lemma \ref{lem:HK-Survival-bound} we get
\begin{align*}
II_{2} 
 & =\int_{D\cap B\left(x,R_{2}\right)^{c}}p_{D}\left(t,x,y\right)d\mu\left(y\right)  \leq C\sup_{z\in D}\mathbb{P}_{z}\left(\tau_{D}>\frac{t}{2}\right)\int_{B\left(x,R_{2}\right)^{c}}p\left(t,x,y\right)d\mu\left(y\right)
\end{align*}
for some constant $C>0$. Applying the bound \eqref{eq:Prop-Main-Est-1} from Proposition \ref{prop:Prop-Main-Est-1} and estimate \eqref{eq:Theorem-proof-II-bound}, we find
\begin{align}\label{eq:Main-Improved-2}
    II_{2}\leq C_1\left(1+\frac{1}{\epsilon}\right)\exp\left({-\frac{1}{2}\frac{1}{1+\frac{d_{1}}{2\alpha_{1}}}\lambda\left(D\right)\left(1-\epsilon\right)t}\right)\frac{t}{\phi\left(R_{2}\right)}.
\end{align}
Combining \eqref{eq:Main-Improved-1} and \eqref{eq:Main-Improved-2} we obtain
\begin{align*}
   & \mathbb{P}_{x}\left(\tau_{D}>t\right) \\ 
   & \leq  C\left(1+\left(\frac{R_{2}}{\phi^{-1}\left(2\epsilon t\right)}\right)^{d_{1}/2}\right)e^{-\lambda\left(D\right)\left(1-\epsilon\right)t}+C_1\left(1+\frac{1}{\epsilon}\right) \exp\left({-\frac{1}{2}\frac{1}{1+\frac{d_{1}}{2\alpha_{1}}}\lambda\left(D\right)\left(1-\epsilon\right)t}\right)\frac{t}{\phi\left(R_{2}\right)},
\end{align*}

As in Proposition \ref{prop:Prop-Main-Est-1}, we choose $R_{2}=\phi^{-1}\left(2\epsilon te^{A_{2}t\lambda\left(D\right)}\right)$
for some $A_{2}>0$ to be decided later. Substituting this in, we obtain 
\begin{align}
 & \mathbb{P}_{x}\left(\tau_{D}>t\right)\nonumber \\
 & \leq C\left(1+e^{\frac{A_{2}d_1}{2\alpha_1}\lambda(D)t}\right)e^{-\lambda\left(D\right)\left(1-\epsilon\right)t}\nonumber \\
 & +C_1\left(1+\frac{1}{\epsilon}\right)\exp\left({-\frac{1}{2}\frac{1}{1+\frac{d_{1}}{2\alpha_{1}}}\lambda\left(D\right)\left(1-\epsilon\right)t}\right)\frac{1}{2\epsilon}e^{-A_{2}\lambda\left(D\right)t}.   \label{eq:Main-Improved-3}
\end{align}

Choose $A_{2}$ so that $\frac{A_{2}d_1}{2\alpha_1}-\left(1-\epsilon\right)=-\frac{\left(1-\epsilon\right)}{1+\frac{d_1}{2\alpha_1}}-A_{2}$
which yields 
\[
A_{2}=\left(-\frac{1}{2}\frac{1}{\left(\frac{d_1}{2\alpha_1}+1\right)^{2}}+\frac{1}{\left(\frac{d_1}{2\alpha_1}+1\right)}\right)\left(1-\epsilon\right).
\]
Plugging into $\left(\ref{eq:Main-Improved-3}\right)$ gives
\begin{equation}\label{eq:proof-iteration-2}
\mathbb{P}_{x}\left(\tau_{D}>t\right)\leq C_2\left(1+\frac1{\epsilon}\right)^2\exp\left(-\left(\frac{1}{\left(1+\frac{d_1}{2\alpha_1}\right)}+\frac{\frac{d_1}{4\alpha_1}}{\left(1+\frac{d_1}{2\alpha_1}\right)^{2}}\right)\left(1-\epsilon\right)\lambda\left(D\right)t\right).
\end{equation}

By iterating the same argument that was made to obtain \eqref{eq:proof-iteration-2}, we can use induction to prove that for all $n \in \mathbb{N}$, we have that there exists a $C_n>0$ such that 
\begin{equation}
\mathbb{P}_{x}\left(\tau_{D}>t\right)\leq C_{n}\left(1+\frac{1}{\epsilon}\right)^{n}\exp\left(-\frac{1-\left(\frac{\frac{d_{1}}{4\alpha_{1}}}{1+\frac{d_{1}}{2\alpha_{1}}}\right)^{n}}{1+\frac{d_{1}}{4\alpha_{1}}}\left(1-\epsilon\right)\lambda\left(D\right)t\right).\label{eq:Induction-1}
\end{equation}

We include the details here for completeness. We assume $\left(\ref{eq:Induction-1}\right)$ holds for $n=k$
and we prove it for $n=k+1$. As before, we split the probability
as 
\begin{align*}
 & \mathbb{P}_{x}\left(\tau_{D}>t\right)\\
 & =\mathbb{P}_{x}\left(\tau_{D}>t,X_{t}\in B\left(x,R_{k+1}\right)\right)+\mathbb{P}_{x}\left(\tau_{D}>t,X_{t}\notin B\left(x,R_{k+1}\right)\right)=:I_{k+1}+II_{k+1}.
\end{align*}
The bound of
\begin{equation}
I_{k+1}\leq C\left(1+\left(\frac{R_{k+1}}{\phi^{-1}\left(2\epsilon t\right)}\right)^{d_{1}/2}\right)e^{-\lambda\left(D\right)\left(1-\epsilon\right)t}\label{eq:Induction-2}
\end{equation}
 is done similarly as in the Step 1 of the proof of Proposition \ref{prop:Prop-Main-Est-1}.
Using Lemma \ref{lem:HK-Survival-bound} combined with $\left(\ref{eq:Induction-1}\right)$ for $n=k$  and by \eqref{eq:Theorem-proof-II-bound}
we get
\begin{align*}
II_{k+1} & =\int_{D\cap B\left(x,R_{k+1}\right)^{c}}p_{D}\left(t,x,y\right)d\mu\left(y\right)\\
 & \leq C\sup_{z\in D}\mathbb{P}_{z}\left(\tau_{D}>\frac{t}{2}\right)\int_{B\left(x,R_{k+1}\right)^{c}}p\left(t,x,y\right)d\mu\left(y\right)\\
 & \leq C_{k}C\left(1+\frac{1}{\epsilon}\right)^{k}\exp\left(-\frac{1}{2}\frac{1-\left(\frac{\frac{d_{1}}{4\alpha_{1}}}{1+\frac{d_{1}}{2\alpha_{1}}}\right)^{k}}{1+\frac{d_{1}}{4\alpha_{1}}}\left(1-\epsilon\right)\lambda\left(D\right)t\right)\frac{t}{\phi\left(R_{k+1}\right)}.
\end{align*}
 Using this with the bound on $I_{k+1}$ we have that 
\begin{align*}
\mathbb{P}_{x}\left(\tau_{D}>t\right) & \leq C\left(1+\left(\frac{R_{k+1}}{\phi^{-1}\left(2\epsilon t\right)}\right)^{d_{1}/2}\right)e^{-\lambda\left(D\right)\left(1-\epsilon\right)t}\\
 & +C_{k}C\left(1+\frac{1}{\epsilon}\right)^{k}\exp\left(-\frac{1}{2}\frac{1-\left(\frac{\frac{d_{1}}{4\alpha_{1}}}{1+\frac{d_{1}}{2\alpha_{1}}}\right)^{k}}{1+\frac{d_{1}}{4\alpha_{1}}}\left(1-\epsilon\right)\lambda\left(D\right)t\right)\frac{t}{\phi\left(R_{k+1}\right)}.
\end{align*}
We choose $R_{k+1}=\phi^{-1}\left(2\epsilon te^{A_{k+1}t\lambda\left(D\right)}\right)$
for some $A_{k+1}>0$ so that we have
\begin{align}
\mathbb{P}_{x}\left(\tau_{D}>t\right) & \leq C\left(1+e^{\frac{A_{k+1}d_{1}}{2\alpha_1}\lambda(D)t}\right)e^{-\lambda\left(D\right)\left(1-\epsilon\right)t}\nonumber \\
 & +C_{k}C\left(1+\frac{1}{\epsilon}\right)^{k}\exp\left(-\frac{1}{2}\frac{1-\left(\frac{\frac{d_{1}}{4\alpha_{1}}}{1+\frac{d_{1}}{2\alpha_{1}}}\right)^{k}}{1+\frac{d_{1}}{4\alpha_{1}}}\left(1-\epsilon\right)\lambda\left(D\right)t\right)\frac{1}{2\epsilon}e^{-A_{k+1}t\lambda\left(D\right)}.\label{eq:Induction-3}
\end{align}
Choose  $A_{k+1}$ such that $\frac{A_{k+1}d_{1}}{2\alpha_{1}}-\left(1-\epsilon\right)=-\frac{1}{2}\frac{1-\left(\frac{\frac{d_{1}}{4\alpha_{1}}}{1+\frac{d_{1}}{2\alpha_{1}}}\right)^{k}}{1+\frac{d_{1}}{4\alpha_{1}}}\left(1-\epsilon\right)-A_{k+1}$. 
We then obtain
\[
A_{k+1}=\left(-\frac{1}{2}\frac{1-\left(\frac{\frac{d_{1}}{4\alpha_{1}}}{1+\frac{d_{1}}{2\alpha_{1}}}\right)^{k}}{\left(1+\frac{d_{1}}{4\alpha_{1}}\right)\left(\frac{d_{1}}{2\alpha_{1}}+1\right)}+\frac{1}{\left(\frac{d_{1}}{2\alpha_{1}}+1\right)}\right)\left(1-\epsilon\right).
\]
With this choice, the exponents   in \eqref{eq:Induction-3} can be combined and become 
\begin{align*}
 & -\frac{1}{2}\frac{1-\left(\frac{\frac{d_{1}}{4\alpha_{1}}}{1+\frac{d_{1}}{2\alpha_{1}}}\right)^{k}}{\left(1+\frac{d_{1}}{4\alpha_{1}}\right)}\left(1-\epsilon\right)\lambda\left(D\right)t-A_{k+1}t\lambda\left(D\right)\\
 & =\left(-\frac{1-\left(\frac{\frac{d_{1}}{4\alpha_{1}}}{1+\frac{d_{1}}{2\alpha_{1}}}\right)^{k+1}}{\left(1+\frac{d_{1}}{4\alpha_{1}}\right)}\right)\left(1-\epsilon\right)\lambda\left(D\right)t.
\end{align*}
Therefore we get that there
exists some constant $C_{k+1}>0$ such that
\begin{align*}
 & \mathbb{P}_{x}\left(\tau_{D}>t\right)\\
 & \leq C_{k+1}\left(1+\frac{1}{\epsilon}\right)^{k+1}\exp\left(\left(-\frac{1-\left(\frac{\frac{d_{1}}{4\alpha_{1}}}{1+\frac{d_{1}}{2\alpha_{1}}}\right)^{k+1}}{\left(1+\frac{d_{1}}{4\alpha_{1}}\right)}\right)\left(1-\epsilon\right)\lambda\left(D\right)t\right),
\end{align*}
as needed.

This proves \eqref{eq:Induction-1} for all $n\in \mathbb{N}$. 
For arbitrary $\eta>0$, we can pick $n$ big enough and $\epsilon>0$ small enough such that $\left(1-\left(\frac{\frac{d_1}{4\alpha_1}}{1+\frac{d_1}{2\alpha_1}}\right)^{n+1}\right)(1-\epsilon)>1-\eta$. Therefore we obtain 
\begin{align}\label{eq-C-eta}
    \mathbb{P}_{x}\left(\tau_{D}>t\right)\le C_\eta \exp \left(-\frac{1-\eta}{1+\frac{d_1}{4\alpha_1}}\lambda\left(D\right)t\right)
\end{align}
for some constant $C_\eta>0$ that depends on $\eta$. Switching $\eta$ back to $\epsilon$ we then get the desired upper bound in \eqref{eq:MainTheorem1-1}.
\end{proof}
This completes the proof of Theorem \ref{thm:MainTheorem1}.

\section{Proof of  Theorem \ref{Thm2} }\label{sec4}

The proof of \eqref{eq:current-best-asymp} follows directly from \eqref{eq:MainTheorem1-1}. The rest of this section is devoted to proving inequality $(\ref{eq:MainTheorem1-2})$. We apply a similar argument as in the proof of the local case in \cite{Mariano-Wang-2022} (see also \cite{Sznitman-1998}). We begin with the following Lemma. 
\begin{lemma}\label{lem:Asymp-Lower-2}
Let $\left(M,d,\left(\mathcal{E},\mathcal{F}\right), \mu\right)$ be a Dirichlet metric measure space as in Theorem \ref{Thm2}. Suppose that for a domain $D\subset M$ the Dirichlet heat kernel $p_{D}$ exists, is continuous, and strictly positive. Then for any $f\in\mathcal{F}\cap C_{0}\left(D\right)$ with $f\geq0$, and for almost every $x\in D,t>0$, there exists a $T>0$ such that
\[
\left(P_{t}^{D}f,f\right)\leq\frac{\left\Vert f\right\Vert _{\infty}^{2}}{\inf_{z\in\text{supp}f}p_{D}\left(x,z,T\right)}P_{t+T}^{D}1\left(x\right).
\]
\end{lemma}

\begin{proof}
By the compactness of  $\text{supp}f\subset D$ and the continuity and strict positivity of $p_{D}$, we have for all $T>0, x\in D$ and $z \in \text{supp}f$ that
\[
0<\inf_{w\in\text{supp}f}p_{D}\left(x,w,T\right)\leq p_{D}\left(x,z,T\right).
\]
Therefore, for all $x\in D$ and $t>0$ we have that
\begin{align*}
\left(P_{t}^{D}f,f\right) = & \int_{D}\int_{D}f\left(y\right)f\left(z\right)p_{D}\left(z,y,t\right)d\mu\left(z\right)d\mu\left(y\right)\\
\leq & \left\Vert f\right\Vert _{\infty}\int_{D}\int_{D}\frac{p_{D}\left(x,z,T\right)}{\inf_{w\in\text{supp}f}p_{D}\left(x,w,T\right)}f\left(z\right)p_{D}\left(z,y,t\right)d\mu\left(z\right)d\mu\left(y\right)\\
\end{align*}
Using the Chapman-Kolmogorov identity we then obtain
\begin{align*}
\left(P_{t}^{D}f,f\right) \le & \frac{\left\Vert f\right\Vert _{\infty}^{2}}{\inf_{w\in\text{supp}f}p_{D}\left(x,w,T\right)}\int_{D}p_{D}\left(x,y,t+T\right)d\mu\left(y\right)\\
= & \frac{\left\Vert f\right\Vert _{\infty}^{2}}{\inf_{w\in\text{supp}f}p_{D}\left(x,w,T\right)}P_{t+T}^{D}1\left(x\right)
\end{align*}
which completes the proof.
\end{proof}
We then prove $\left(\ref{eq:MainTheorem1-2}\right)$ in the following
proposition. 

\begin{prop}\label{prop:Asymp-Lower-3}
Let $\left(M,d,\left(\mathcal{E},\mathcal{F}\right), \mu\right)$ be a Dirichlet metric measure space as in Theorem \ref{Thm2}. Suppose that for a domain $D\subset M$ the Dirichlet heat kernel $p_{D}$ exists, is continuous, and strictly positive. If $\lambda\left(D\right)>0$, then for all $x\in D$,
\begin{equation}\label{eq:Lower_bound_in_Prop}
\liminf_{t\to\infty}\frac{\log\mathbb{P}_{x}\left(\tau_{D}>t\right)}{t}\geq-\lambda\left(D\right).
\end{equation}

\end{prop}

\begin{proof}
Let $\epsilon>0$. By the definition of $\lambda\left(D\right)$,
we can find $f_{\epsilon}\in\mathcal{F}\cap C_{0}\left(D\right)$
such that $\left\Vert f_\epsilon\right\Vert _{2}=1$ and
\begin{equation}\label{eq:Lower_proof_Eq1}
\lambda\left(D\right)\leq\mathcal{E}\left(f_{\epsilon},f_{\epsilon}\right)\leq\lambda\left(D\right)+\epsilon,
\end{equation}
Therefore we have
\begin{align}\label{eq-exp-tE1}
e^{-t\left(\lambda\left(D\right)+\epsilon\right)} \leq e^{-t\mathcal{E}\left(f_{\epsilon},f_{\epsilon}\right)}  & \leq  e^{-t\mathcal{E}\left(\left|f_{\epsilon}\right|,\left|f_{\epsilon}\right|\right)} 
\end{align}
where the second inequality follows from the Markov property of the Dirichlet form $(\mathcal{E},\mathcal{F}_D)$. 

Moreover, for any $g\in\mathcal{F}\cap C_{0}\left(D\right)$ where $\left\Vert g\right\Vert _{2}=1$
the Spectral Theorem and Jensen's inequality yield 
\begin{align}\label{eq:Lower_proof_Eq2}
e^{-t\mathcal{E}\left(g,g\right)}  =\exp\left(-t\int_{0}^{\infty}\lambda d\left\Vert E_{\lambda}g\right\Vert _{2}^{2}\right) \leq\int_{0}^{\infty}e^{-t\lambda}d\left\Vert E_{\lambda}g\right\Vert _{2}^{2} =\left(P_{t}g,g\right).
\end{align}
Applying \eqref{eq:Lower_proof_Eq2}  to $|f_\epsilon|$ and invoking Lemma \ref{lem:Asymp-Lower-2} we get
\begin{align}\label{eq-exp-tE2}
e^{-t\mathcal{E}\left(\left|f_{\epsilon}\right|,\left|f_{\epsilon}\right|\right)} 
\leq\left(P_{t}\left|f_{\epsilon}\right|,\left|f_{\epsilon}\right|\right)
\leq\frac{\left\Vert f_{\epsilon}\right\Vert _{\infty}^{2}}{\inf_{z\in\text{supp}\left|f_{\epsilon}\right|}p_{D}\left(x,z,T\right)}P_{t+T}^{D}1\left(x\right).
\end{align} 
Combining \eqref{eq-exp-tE1} and \eqref{eq-exp-tE2} we obtain for all $x\in D$, $t,T>0$ that
\[
P_{t+T}^{D}1\left(x\right)\geq\frac{\inf_{z\in\text{supp}\left|f_{\epsilon}\right|}p_{D}\left(x,z,T\right)}{\left\Vert f_{\epsilon}\right\Vert _{\infty}^{2}}e^{-t\left(\lambda\left(D\right)+\epsilon\right)}.
\]
The continuity of $p_{D}$ ensures that $P_{t+T}^{D}1\left(x\right)=\mathbb{P}_{x}\left(\tau_{D}>t+T\right)$ for all $x\in D$. Hence the above lower bound also applies to $\mathbb{P}_{x}\left(\tau_{D}>t+T\right)$. Since for any $T>0$, we have $\mathbb{P}_{x}\left(\tau_{D}>t\right)\geq\mathbb{P}_{x}\left(\tau_{D}>t+T\right)$
then for any $x\in D$, 
\[
\liminf_{t\to\infty}\frac{\log\mathbb{P}_{x}\left(\tau_{D}>t\right)}{t}\geq-\lambda\left(D\right)-\epsilon.
\]
The desired conclusion follows by letting $\epsilon\to0$.

\end{proof}

\section{Proofs of the results from Section \ref{sec:1}}\label{sec5}

\begin{proof}[Proof of Theorem \ref{Main:Thm:StableProc}]
The upper bounds follow by \eqref{eq:MainTheorem1-1} applied to $\alpha$-stable processes in $\mathbb{R}^d$ when $0<\alpha<2$.

The lower bound asymptotic in \eqref{eq:Main-Simplified-Asymp-1} follows by \eqref{eq:MainTheorem1-2} once we verify the conditions of that theorem. By \cite{Chen-Song-1997} it is well known that for any
domain $D\subset\mathbb{R}^{d}$ the Dirichlet heat kernel $p_{D}$
exists, is continuous, and strictly positive. Then by $(\ref{eq:MainTheorem1-2})$
in Theorem \ref{thm:MainTheorem1} we have that
\[
-\lambda\left(D\right)\leq\liminf_{t\to\infty}\frac{\log\mathbb{P}_{x}\left(\tau_{D}>t\right)}{t}.
\]
    
\end{proof}

\begin{proof}[Proof of  Corollary \ref{cor:Equivalences-intro-general} and Corollary \ref{cor:Equivalences-gen}]
Suppose $\lambda\left(D\right)>0$. Then by $\left(\ref{eq:Main-Simplified-Upper-1}\right)$
from Theorem \ref{Main:Thm:StableProc}, there exists a constant $C>0$
such that
\begin{equation}\label{eq:ProofMainCor-1}
\mathbb{E}_{x}\left[\tau_{D}\right]=\int_{0}^{\infty}\mathbb{P}_{x}\left(\tau_{D}>t\right)dt\leq\frac{C}{\lambda\left(D\right)}<\infty,\
\end{equation}
for any $x\in D$. This implies that $\sup_{x\in D}\mathbb{E}_{x}\left[\tau_{D}\right]<\infty$.
Conversely, if $\sup_{x\in D}\mathbb{E}_{x}\left[\tau_{D}\right]<\infty$, it is well known that $\lambda\left(D\right)\geq\frac{1}{\sup_{x\in D}\mathbb{E}_{x}\left[\tau_{D}\right]}>0$.
(See discussion in \cite[Remark 3.4.]{Mariano-Wang-2022} for various references).

The proof of Corollary \ref{cor:Equivalences-gen} is done similarly using the Theorem  \ref{thm:MainTheorem1} for the upper bound and \cite[Remark 3.4.]{Mariano-Wang-2022}  for the lower bound. 
\end{proof}


\begin{proof}[Proof of  Theorem \ref{cor:Equivalences}]
\noindent\underline{\textit{Part (1):}}
The proof follows a similar argument as in \cite[Theorem 4.6]{Barrier-Bogdan-Grzywny-Ryznar-2015}. Suppose $D$ is $C^{1,1}$ at a scale $r>0$.
Let $x\in D$, and pick a point $q\in \partial D$ such that $\delta_{D}\left(x\right)=\left|x-q\right|$. By definition of $C^{1,1}$ this means 
there exists $x_{in}\in D,x_{out}\in D^{c}$ with $B\left(x_{in},r\right)\subset D$
and $B\left(x_{out},r\right)\subset D^{c}$ such that $\overline{B\left(x_{in},r\right)}\cap\overline{B\left(x_{out},r\right)}=\left\{ q\right\} $.

We divide the proof into two cases.

\textbf{Case 1:} Suppose $\delta_{D}\left(x\right)\leq\frac{r}{2}$.
Define $D\left(q\right)=D\cap B\left(q,r\right)$. By the strong Markov
property we have
\begin{align*}
\mathbb{E}_{x}\left[\tau_{D}\right] & =\mathbb{E}_{x}\left[\tau_{D\left(q\right)}\right]+\mathbb{E}_{x}\left[\mathbb{E}_{X_{\tau_{D\left(q\right)}}}\left[\tau_{D}\right]1_{\left\{ \tau_{D}>\tau_{D\left(q\right)}\right\} }\right].
\end{align*}
Note that 
\begin{align*}
\mathbb{E}_{x}\left[\mathbb{E}_{X_{\tau_{D\left(q\right)}}}\left[\tau_{D}\right]1_{\left\{ \tau_{D}>\tau_{D\left(q\right)}\right\} }\right] & \leq\mathbb{E}_{x}\left[\mathbb{E}_{X_{\tau_{D\left(q\right)}}}\left[\tau_{D}\right]1_{\left\{ \left|X_{\tau_{D\left(q\right)}}-q\right|\geq r\right\} }1_{\left\{ X_{\tau_{D\left(q\right)}}\in D\right\} }\right]\\
 & \leq\sup_{y\in D}\mathbb{E}_{y}\left[\tau_{D}\right]\mathbb{P}_{x}\left(\left|X_{\tau_{D\left(q\right)}}-q\right|>r\right),
\end{align*}
so that 
\[
\mathbb{E}_{x}\left[\tau_{D}\right]\leq\mathbb{E}_{x}\left[\tau_{D\left(q\right)}\right]+\sup_{y\in D}\mathbb{E}_{y}\left[\tau_{D}\right]\mathbb{P}_{x}\left(\left|X_{\tau_{D\left(q\right)}}-q\right|>r\right).
\]
By \cite[Corollary 2.8]{Barrier-Bogdan-Grzywny-Ryznar-2015} we know that $\mathbb{P}_{x}\left(\left|X_{\tau_{D(q)}}-q\right|\geq r\right)\leq C_{d}\frac{\mathbb{E}_{x}\left[\tau_{D(q)}\right]}{r^{\alpha}}$. Combining this with estimate \eqref{eq:ProofMainCor-1}, 
\[
\mathbb{E}_{x}\left[\tau_{D}\right]\leq\mathbb{E}_{x}\left[\tau_{D(q)}\right]+\frac{C}{\lambda\left(D\right)}C_{d}\frac{\mathbb{E}_{x}\left[\tau_{D(q)}\right]}{r^{\alpha}}.
\]
From \cite[Corollary 4.5]{Barrier-Bogdan-Grzywny-Ryznar-2015}, $\mathbb{E}_{x}\left[\tau_{D(q)}\right]\leq C_rr^{\alpha/2}\delta_{D}\left(x\right)^{\alpha/2}$. Therefore
\[
\mathbb{E}_{x}\left[\tau_{D}\right]
\leq C_{D,r,d}\delta_{D}\left(x\right)^{\alpha/2},
\]
for some $ C_{D,r,d}>0$, as desired. 
 
 The lower bound follows from $\mathbb{E}_{x}\left[\tau_{D}\right]\geq\mathbb{E}_{x}\left[\tau_{B\left(x_{in},r\right)}\right]\geq C_{d,\alpha}\delta_{D}\left(x\right)^{\alpha/2}r^{\alpha/2}$ where we used \cite[Theorem 4.1]{Bogdan-Grzywny-Ryznar-2010}.

\textbf{Case 2:} Suppose $\delta_{D}(x)\geq\frac{r}{2}$. From $\left(\ref{eq:ProofMainCor-1}\right)$, 
\[
\mathbb{E}_{x}\left[\tau_{D}\right]\leq\frac{C}{\lambda\left(D\right)}\leq\frac{C}{\lambda\left(D\right)}\frac{\delta_{D}(x)^{\alpha/2}}{r^{\alpha/2}},
\]
which gives the desired upper bound.
For the lower bound, note that $B\left(x,\delta_{D}(x)\right)\subset D$. Thus $\mathbb{E}_{x}\left[\tau_{D}\right]\geq\mathbb{E}_{x}\left[\tau_{B\left(x,\delta_{D}\left(x\right)\right)}\right]\geq C\delta_{D}\left(x\right)^{\alpha}\geq C\delta_{D}(x)^{\alpha/2}r^{\alpha/2}$. \\

\noindent\underline{\textit{Part (2):}}
Assume $R_D<\infty$. The implication in \eqref{eq:lambda-lower} was established in \cite[Remark 4.1]{Bianchi-Brasco-2022} for domains under a mild regularity condition. We adapt that argument to $C^{1,1}$ domains. 

We first recall the following equivalent definition of $\lambda\left(D\right)$
using the Gagliardo--Slobodecki\u{\i} seminorm
\[
\left[u\right]_{W^{\frac{\alpha}{2},2}\left(\mathbb{R}^{d}\right)}^{2}=\frac{c_{d,\alpha}}{2}\int\int_{\mathbb{R}^{d}\times\mathbb{R}^{d}}\frac{\left|u\left(x\right)-u\left(y\right)\right|^{2}}{\left|x-y\right|^{d+\alpha}}dxdy,
\]
by writing 
\begin{equation}
\lambda\left(D\right)=\inf_{u\in\tilde{W}_{0}^{\alpha/2,2}\left(D\right)\backslash\left\{ 0\right\} }\frac{\left[u\right]_{W^{\frac{\alpha}{2},2}\left(\mathbb{R}^{d}\right)}^{2}}{\left\Vert u\right\Vert _{L^{2}\left(D\right)}^{2}},\label{eq:lambda-equiv1}
\end{equation}
where $\tilde{W}_{0}^{\alpha/2,2}\left(D\right)$ is the closure of
$C_{0}^{\infty}\left(D\right)$ in the fractional Sobolev space 
\[
W^{\frac{\alpha}{2},2}\left(\mathbb{R}^{d}\right):=\left\{ u\in L^{2}\left(\mathbb{R}^{d}\right):\left[u\right]_{W^{\frac{\alpha}{2},2}\left(\mathbb{R}^{d}\right)}<\infty\right\} ,
\]
which is endowed with the norm 
\[
\left\Vert u\right\Vert _{W^{\frac{\alpha}{2},2}\left(\mathbb{R}^{d}\right)}=\left\Vert u\right\Vert _{L^{2}\left(\mathbb{R}^{d}\right)}+\left[u\right]_{W^{\frac{\alpha}{2},2}\left(\mathbb{R}^{d}\right)}.
\]
The infimum in $\left(\ref{eq:lambda-equiv1}\right)$ can be equivalently
taken in $C_{0}^{\infty}\left(D\right)$ as in \eqref{eq-lamdD}.

Now fix $x\in D$ and choose $q\in\partial D$ such that $\delta_{D}\left(x\right)=|x-q|$.
Since $D$ is uniformly $C^{1,1}$, there
exists $p_{out}\in D^{c}$ such that $B\left(p_{out},r\right)\subset D^{c}$
and is tangent to $\partial D$ at $q$. For any $y\in B\left(p_{out},r\right)$,
\begin{align*}
\left|x-y\right| & \leq\left|x-q\right|+\left|q-y\right|\leq\delta_{D}\left(x\right)+2r\leq R_{D}+2r.
\end{align*}
Hence for any $u\in C_{0}^{\infty}\left(D\right)\cap \nH$ we have 
\begin{align*}
&\int\int_{\mathbb{R}^{d}\times\mathbb{R}^{d}}\frac{\left|u\left(x\right)-u\left(y\right)\right|^{2}}{\left|x-y\right|^{d+\alpha}}dxdy \geq\int_{D}\left(\int_{B\left(p_{out},r\right)}\frac{\left|u\left(x\right)\right|^{2}}{\left|x-y\right|^{d+\alpha}}dy\right)dx\\
 & \quad\quad\geq\frac{1}{\left(R_{D}+2r\right)^{d+\alpha}}\int_{D}\left|B\left(p_{out},r\right)\right|\left|u\left(x\right)\right|^{2}dx =\frac{\omega_{d}r^{d}}{\left(R_{D}+2r\right)^{d+\alpha}}\int_{D}\left|u\left(x\right)\right|^{2}dx.
\end{align*}
This yields \eqref{eq:lambda-lower}, as desired.

Finally, we observe that  \eqref{eq:C11-equiv} follows directly from \eqref{eq:Mean-bound} and \eqref{eq:lambda-lower}.
\end{proof}

Next we prove the Dirichlet heat kernel upper bound claimed in  Proposition \ref{prop-kernel-bd}.
\begin{proof}[Proof of Proposition \ref{prop-kernel-bd}]

Suppose $D$ is any convex domain in $\mathbb{R}^{d}$. If $\lambda(D)=0$, then the proof is easy, so we assume $\lambda(D)>0$.  The proof
of the heat kernel upper bound $\left(\ref{eq:HK-up}\right)$ follows
similarly as in the proof given by Siudeja in \cite[Theorem 1.6]{Siudeja-2006}. The main difference is
that Lemma 3.4 in \cite{Siudeja-2006} is replaced by Lemma \ref{lem:HK-Survival-bound}
with the bound on $\mathbb{P}_{x}\left(\tau_{D}>t\right)$ given by
Theorem \ref{Main:Thm:StableProc}. We include some details here for completeness. 

When $D$ is a convex domain,  \cite[Proposition 3.3]{Siudeja-2006}  shows that 
\begin{equation}
p_{D}\left(t,x,y\right)\leq c\left(1\wedge\frac{\delta_{D}\left(y\right)^{\alpha/2}}{\sqrt{t}}\right)p\left(t,x,y\right).\label{eq:HK-proof-1}
\end{equation}
By Lemma \ref{lem:HK-Survival-bound} combined with the upper bound in Theorem \ref{Main:Thm:StableProc} 
there exists a constant $C_{\eta}>0$ such that 
\begin{align}
p_{D}\left(t,x,y\right) & \leq C\sup_{z\in D}\mathbb{P}_{z}\left(\tau_{D}>\frac{t}{2}\right)p\left(t,x,y\right)\leq C_{\eta}e^{-\frac{1}{2}\frac{\left(1-\eta\right)}{1+\frac{d}{4\alpha}}\lambda\left(D\right)t}p\left(t,x,y\right).\label{eq:HK-proof-2}
\end{align}
Using Chapman-Kolmogorov, and $\left(\ref{eq:HK-proof-1}\right)$
with $\left(\ref{eq:HK-proof-2}\right)$ we have 
\begin{align*}
p_{D}\left(t,x,y\right) & =\int_{D}p_{D}\left(\left(1-\eta\right)t,x,z\right)p_{D}\left(\eta t,z,y\right)dz\\
 & \leq\int_{D}C_{\eta}e^{-\frac{1}{2}\frac{\left(1-\eta\right)^{2}}{1+\frac{d}{4\alpha}}\lambda\left(D\right)t}p\left(\left(1-\eta\right)t,x,z\right)c\left(1\wedge\frac{\delta_{D}\left(y\right)^{\alpha/2}}{\sqrt{t}}\right)p\left(\eta t,z,y\right)dz\\
 & \leq C_{\eta}e^{-\frac{1}{2}\frac{\left(1-\eta\right)^{2}}{1+\frac{d}{4\alpha}}\lambda\left(D\right)t}\left(1\wedge\frac{\delta_{D}\left(y\right)^{\alpha/2}}{\sqrt{t}}\right)p\left(t,x,y\right).
\end{align*}
We can choose $\eta$ so small so that $\left(1-\eta\right)^{2}=1-\epsilon$
so that there exists a constant $C_{\epsilon}>0$ so that 
\begin{equation}
p_{D}\left(t,x,y\right)\leq C_{\epsilon}e^{-\frac{1}{2}\frac{\left(1-\epsilon\right)}{1+\frac{d}{4\alpha}}\lambda\left(D\right)t}\left(1\wedge\frac{\delta_{D}\left(y\right)^{\alpha/2}}{\sqrt{t}}\right)p\left(t,x,y\right)\label{eq:HK-proof-3}
\end{equation}
Using another application of Chapman-Kolmogorov, symmetry and the
bound of $\left(\ref{eq:HK-proof-3}\right)$ we get
\begin{align*}
 & p_{D}\left(t,x,y\right)=\int_{D}p_{D}\left(\frac{t}{2},x,z\right)p_{D}\left(\frac{t}{2},z,y\right)dz\\
 & =\int_{D}p_{D}\left(\frac{t}{2},z,x\right)p_{D}\left(\frac{t}{2},z,y\right)dz\\
 & \leq C_{\epsilon}^{2}\int_{D}\left(e^{-\frac{1}{2}\frac{\left(1-\epsilon\right)}{1+\frac{d}{4\alpha}}\lambda\left(D\right)\frac{t}{2}}\right)^{2}\left(1\wedge\frac{\delta_{D}\left(x\right)^{\alpha/2}}{\sqrt{t}}\right)p\left(\frac{t}{2},z,x\right)\left(1\wedge\frac{\delta_{D}\left(y\right)^{\alpha/2}}{\sqrt{t}}\right)p\left(\frac{t}{2},z,y\right)dz\\
 & =C_{\epsilon}^{2}e^{-\frac{1}{2}\frac{\left(1-\epsilon\right)}{1+\frac{d}{4\alpha}}\lambda\left(D\right)t}\left(1\wedge\frac{\delta_{D}\left(x\right)^{\alpha/2}}{\sqrt{t}}\right)\left(1\wedge\frac{\delta_{D}\left(y\right)^{\alpha/2}}{\sqrt{t}}\right)p\left(t,x,y\right),
\end{align*}
as desired. 

\end{proof}

\begin{proof}[Proof of Theorem \ref{thm:Horn-Domains-1}]

\medskip

We first show that 
$$\lambda\left(H\right)=\lambda_1\left(h\right).  $$
Let us fix notation. Consider any bounded connected open set $\omega\subset\mathbb{R}^{d-1}$, and let $\mathcal{L}_{\omega}$ denote the generator of the   $\left(d-1\right)-$dimensional
$\alpha-$ stable process killed upon exiting $\omega$. It is well-known that $\mathcal{L}_{\omega}$
has discrete spectrum. We denote its first eigenvalue by $\lambda_{1}\left(\omega\right)$.
Consider in $\mathbb{R}^{d}$ the semi-tube 
\begin{align*}
Q_{a}\left(\omega\right) & :=(a,\infty)\times\omega=\{x=(x_1,z): x_1> a, z\in \omega\},\quad a\in\R.
\end{align*}
The full tube is denoted by
\begin{align*}
Q_{-\infty}\left(\omega\right) & :=(-\infty,\infty)\times\omega=\{x=(x_1,z): x_1\in\R, z\in \omega\}.
\end{align*}
Let $\mathcal{L}_{Q_{a}\left(\omega\right)}$ 
denote  the generator of the killed $\alpha-$ stable processes on $Q_{a}\left(\omega\right)$, $a\in\R\cup\{-\infty\}$. 
By \cite[Corollary 1]{Bakharev-Nazarov-2023},  its spectrum 
coincides with the ray $\left[\lambda_{1}\left(\omega\right),\infty\right)$. Therefore, the bottom of the spectrum $\lambda\left(Q_{a}\left(\omega\right)\right)$ satisfies
\begin{align}\label{eq-lambda-horn}
   \lambda\left(Q_{a}\left(\omega\right)\right)
=\lambda_{1}\left(\omega\right). 
\end{align}
Now let $H$ be the horn-shaped domain as given in the theorem, and let $H(a)$, $a>0$ be a cross section of $H$ at $a$ (see Definition \ref{def-hornshape}). By the  monotonicity of the cross sections of $H$ we can easily see that $H$ is sandwiched by two semitubes:
\[
Q_a(H(a))\subset H \subset Q_0 (h) 
\]
where $h$ is given as in \eqref{eq-h}.
Using the monotonicity of the bottom of the spectrum we then obtain  that 
\begin{align}\label{eq-sandwich}
    \lambda\left(Q_0\left(h\right)\right)\leq\lambda\left(H\right)\leq\lambda\left(Q_{a}(H(a))\right).
\end{align}
for all $a>0$. Note that $Q_a(H(a))$ is a translation of $Q_0(H(a))$, which implies that $\lambda(Q_a(H(a)))=\lambda(Q_0(H(a)))$. Hence when $a\to\infty$, we have that 
\begin{align}\label{eq-sand-1}
\lim_{a\to\infty}\lambda(Q_a(H(a)))=\lim_{a\to\infty}\lambda(Q_0(H(a)))=\lim_{a\to\infty}\lambda_1(H(a))=\lambda_1(h)
\end{align}
where the last equality follows from the fact that $\cup_{a\ge0}H(a)=h$.
Using \eqref{eq-lambda-horn} we also know that 
\begin{align}\label{eq-sand-2}
    \lambda\left(Q_0\left(h\right)\right)=\lambda_1(h).
\end{align}
 Combining \eqref{eq-sandwich}, \eqref{eq-sand-1}, and \eqref{eq-sand-2} we then obtain that
\begin{align}\label{eq-lamb-equal}
  \lambda\left(H\right)=\lambda_1\left(h\right)  
\end{align}

We now prove \eqref{eq:Horn-sharp} by estimating the upper and lower bounds  of $\mathbb{P}_{x}\left(\tau_{H}>t\right)$ separately.  

\medskip
\noindent\underline{\textit{Upper bound estimate:}}
We write the $d-$dimensional symmetric $\alpha$-stable
process $X_{t}$ as $(X^1_t, Z_t)$, where $X^1_t$ and $Z_t$ are subordinated Brownian motions in $\R$ and in $\R^{d-1}$, respectively, both with the subordinator $S_{t}$. Let $\tau_h$ denote the first exit time of the process $Z_t$ from $h$. Then by the definition of $Q_{-\infty}(h)$ for any starting point $x=(x_1,z)\in H$, 
\[
\mathbb{P}_{(x_1,z)}\left(\tau_{Q_{-\infty}\left(h\right)}>t\right)=\mathbb{P}_{z}\left(\tau_{h}>t\right).
\]
Moreover since $H\subset Q_{-\infty}\left(h\right)$, we have for any $t>0$, 
\begin{align*}
\mathbb{P}_{\left(x_1,z\right)}\left(\tau_{H}>t\right) & \leq\mathbb{P}_{\left(x_1,z\right)}\left(\tau_{Q_{-\infty}\left(h\right)}>t\right).
\end{align*}
Combining the above yields 
\[
\limsup_{t\to\infty}\frac{\log\mathbb{P}_{\left(x_1,z\right)}\left(\tau_{H}>t\right)}{t}\leq\limsup_{t\to\infty}\frac{\log\mathbb{P}_{z}\left(\tau_{h}>t\right)}{t}=-\lambda_1\left(h\right).
\]
The last equality follows from classical asymptotic results (see \cite{Kulczycki-1998}) for bounded domains in $\R^{d-1}$. Combining with \eqref{eq-lamb-equal} we obtain
\[
\limsup_{t\to\infty}\frac{\log\mathbb{P}_{x}\left(\tau_{H}>t\right)}{t}
\le -\lambda_1\left(h\right)=-\lambda\left(H\right).
\]
\noindent\underline{\textit{Lower bound estimate:}} 
By \cite{Chen-Song-1997} it is well known that for any
domain $D\subset\mathbb{R}^{d}$ the Dirichlet heat kernel $p_{D}$
exists, is continuous, and strictly positive. Then by $(\ref{eq:MainTheorem1-2})$
in Theorem \ref{thm:MainTheorem1} we have that
\[
-\lambda\left(H\right)\leq\liminf_{t\to\infty}\frac{\log\mathbb{P}_{x}\left(\tau_{H}>t\right)}{t}.
\]
Combining the upper and lower bounds completes the proof. 
\end{proof}

\section{Examples}\label{sec:ex}

The purpose of this section is to illustrate examples where our results apply. In particular, we give emphasis to unbounded domains with a positive bottom of the spectrum. These examples fall outside the scope of classical results, as they do not have discrete spectra, are not inner uniform, and are not necessarily convex.

We start with an example of a Horn-shaped domain. 

\begin{example}[Horn-Shape Example]\label{Ex:HornShaped}
   Let $
H=\left\{ \left(x,y\right)\in\mathbb{R}\times \mathbb{R}^{d-1}\mid x>0,\left|y\right|<\frac{x}{x+1}\right\} $. This is a horn-shaped domain that is convex. See Figure~\ref{fig:horn-shape}. All the results in this paper apply to this domain. In particular, the sharp asymptotic behavior of the survival probability $\mathbb{P}_x(\tau_H > t)$ is given in Theorem~\ref{thm:Horn-Domains-1}:
\begin{equation}\label{eq:Horn-sharp2}
\lim_{t\to\infty}\frac{\log\mathbb{P}_{x}\left(\tau_{H}>t\right)}{t}=-\lambda\left(H\right)=-\lambda_{1}\left(B_{d-1}(0,1)\right),
\end{equation}
where $B_{d-1}(0,1)$ is a $(d-1)$ dimensional ball of radius 1. See \cite{Dyda-etall-2017} for some numerical estimates on $\lambda_{1}\left(B_{d-1}(0,1)\right)$. 

In general the cross sections of a horn-shaped domain $H$ do not need to be convex for \eqref{eq:Horn-sharp2} to hold. A similar result can be shown for similarly defined $H$ with a profile function given by $f(x)=\frac{x}{x+1}$ but with non-convex and non-ball cross-sections.

\end{example}

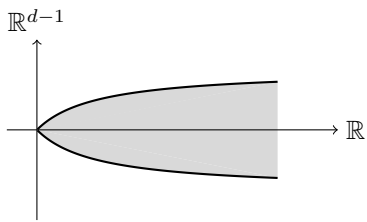
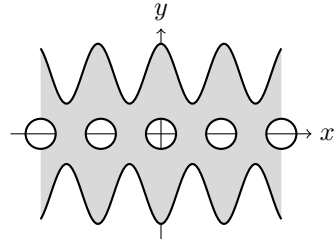
\begin{figure}[ht]
\centering

\begin{subfigure}[t]{0.45\linewidth} 
\centering
\begin{tikzpicture}[scale=0.8] 
  \fill[gray!30] (0,0) -- (4,4/5) -- (4,-4/5) -- cycle;
  \fill[gray!30, domain=0.01:4, variable=\x]
    (0.01,0) -- plot (\x, { \x / (\x + 1) }) --
    plot (\x, { -\x / (\x + 1) }) -- cycle;

  \draw[->] (-0.5,0) -- (5,0) node[right] {$\mathbb{R}$};
  \draw[->] (0,-1.5) -- (0,1.5) node[above] {$\mathbb{R}^{d-1}$};

  \draw[thick, domain=0.01:4, samples=100] plot (\x, { \x / (\x + 1) });
  \draw[thick, domain=0.01:4, samples=100] plot (\x, { -\x / (\x + 1) });
\end{tikzpicture}
\caption{Horn-shaped region from Example \ref{Ex:HornShaped}}
\label{fig:horn-shape}
\end{subfigure}%
\hfill 
\begin{subfigure}[t]{0.45\linewidth}
\centering
\begin{tikzpicture}[scale=0.4]
  \draw[->] (-5,0) -- (5,0) node[right] {$x$};
  \draw[->] (0,-3.5) -- (0,3.5) node[above] {$y$};

  \begin{scope}
    \clip (-5,-3.5) rectangle (5,3.5)
          (-4,0) circle (0.5)
          (-2,0) circle (0.5)
          (0,0) circle (0.5)
          (2,0) circle (0.5)
          (4,0) circle (0.5);
    \fill[gray!30]
      plot[domain=-4:4, samples=100] (\x, { 2 + cos(deg(3*\x)) }) --
      plot[domain=4:-4, samples=100] (\x, { -2 - cos(deg(3*\x)) }) -- cycle;
  \end{scope}

  \draw[thick, domain=-4:4, samples=100] plot (\x, { 2 + cos(deg(3*\x)) });
  \draw[thick, domain=-4:4, samples=100] plot (\x, { -2 - cos(deg(3*\x)) });

  \foreach \xx in {-4,-2,0,2,4} {
    \draw[thick] (\xx,0) circle (0.5);
  }
\end{tikzpicture}
\caption{Unbounded domain from Example \ref{Ex:General}}
\label{fig:ExampleDomain1}
\end{subfigure}

\caption{Illustrations of unbounded domains used in the examples.}
\label{fig:domains}
\end{figure}


The next example is a non-convex and non-simply connected domain that is bounded in the $y$-direction but unbounded in the $x$-direction.  

\begin{example}[Unbounded domain with $\lambda (D)>0$]\label{Ex:General}
  Define 
  \[
U=\left\{ \left(x,y\right)\in\mathbb{R}^{2}\mid\left|y\right|<2+\cos\left(3x\right)\right\} \backslash\bigcup_{i\in\Z}\overline{B}\left((2i,0),\frac{1}{2}\right)
\]
as illustrated in Figure \ref{fig:ExampleDomain1}. The domain $U$ is uniformly $C^{1,1}$ and it satisfies $\lambda (U)>0$. This is a consequence of \eqref{eq:C11-equiv} since 
$\mathbb{E}_x[\tau_U]\le \mathbb{E}_x[\tau_S]<\infty$ where $S:=\left\{ \left(x,y\right)\in\mathbb{R}^{2}\mid\left|y\right|<3\right\}$ is an infinite strip that contains $U$.
However, $U$ does not have a discrete spectrum. This follows from a result of Kwa\'snicki \cite[Lemma 1]{Kwasnicki-2009},  which states that the heat semigroup operator $P_t^D$ is compact if and only if $\mathbb{E}_x[\tau_D]\to 0$ as $|x|\to\infty$. Since this is not the case for $U$, the spectrum is not discrete.  As a result, Theorem \ref{Main:Thm:StableProc} provides the best available description of the long-time behavior of the survival probability in current setting.  
\end{example}

The next example is a non-simply connected domain that is unbounded in every direction. 

\begin{example}[An unbounded Swiss cheese domain]
Let $\mathcal{Q}$ be an unbounded Swiss cheese type domain,  defined as $\mathbb{R}^{d}$ with an infinite array of evenly spaced balls removed. Figure \ref{fig:SwissCheese} illustrates the case when $d=2$.  The domain is of class $C^{1,1}$, but it is neither simply connected nor bounded in any direction. Moreover, by a result of Kwa\'snicki \cite[Lemma 1]{Kwasnicki-2009}, the domain $\mathcal{Q}$ does not admit a discrete spectrum. It is easy to see that its inradius satisfies $R_{\mathcal{Q}}<\infty$. By Theorem \ref{cor:Equivalences} we have that $\lambda \left(\mathcal{Q}\right)>0$, hence Theorem \ref{Main:Thm:StableProc} gives the best known estimate for the long term behavior of the survival probability in $\mathcal{Q}$. 

Here we give an explicit construction in $\mathbb{R}^2$, while a similar construction can be made in $\mathbb{R}^d$. Let
$
\mathring{Q}_{\epsilon}:=\left(\left[0,1\right]\times\left[0,1\right]\right)\backslash \overline{B}\left(\left(\frac{1}{2},\frac{1}{2}\right),\epsilon\right)
$
with a choice of $\epsilon<\frac{1}{2}$. Define the set
\[
\mathcal{Q}=\bigcup_{\left(i,j\right)\in\mathbb{Z}\times\mathbb{Z}}\left(\mathring{Q}_{\epsilon}+\left(i,j\right)\right).
\]
Note that $\mathcal{Q}$ is the plane $\mathbb{R}^{2}$ with equally
spaced circular holes of radius $\epsilon$.  
Note that the inradius is $R_{\mathcal{Q}}=\frac{1}{\sqrt{2}}-\epsilon$.

\end{example}

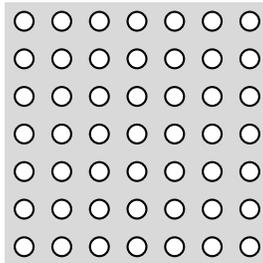
\begin{figure}[h]
    \centering
\begin{tikzpicture}[scale=0.5]
    \fill[gray!30] (-3.5,-3.5) rectangle (3.5,3.5);
    
    \foreach \x in {-3,-2,-1,0,1,2,3} {
        \foreach \y in {-3,-2,-1,0,1,2,3} {
            \fill[white] (\x,\y) circle (0.25);
            \draw[thick] (\x,\y) circle (0.25);
        }
    }
\end{tikzpicture}
   \caption{Infinite ``swiss cheese".
}
    \label{fig:SwissCheese}
\end{figure}

We also point to the work of \cite{bianchi2023spectrumsetscorestubes} for other examples of unbounded domains with positive bottom of the spectrum where our results can be applied.

\begin{acknowledgement}
We would like to thank Rodrigo Ba\~nuelos, Mathav Murugan and Hugo Panzo for helpful discussions. We also thank the anonymous referees whose insightful comments and suggestions have led to an improved manuscript.
\end{acknowledgement}

\bibliographystyle{plain}	
\bibliography{MainRef}

\end{document}